\documentclass[reqno, 12pt]{amsart}

\usepackage[utf8]{inputenc}
\usepackage[T1]{fontenc}
\usepackage{lmodern}
\usepackage{amsmath}
\usepackage{amsthm}
\usepackage{amssymb}
\usepackage{latexsym}
\usepackage{amsmath,amsfonts,amssymb}
\usepackage{upgreek}
\usepackage{enumerate}
\usepackage{mathrsfs}
\usepackage{mathscinet}
\usepackage{stmaryrd} 
\usepackage{graphicx}
\usepackage{xcolor}
\usepackage[all,cmtip]{xy}
\usepackage{url}
\usepackage{hyperref}
\usepackage[a4paper, hmargin=3cm, vmargin=3cm]{geometry}

\newcommand\mathens[1]{\mathbb{#1}} 

\newcommand{\ud}{\mathrm{d}}

\newcommand{\N}{\mathens{N}}
\newcommand{\Z}{\mathens{Z}}

\newcommand{\R}{\mathens{R}}
\newcommand{\C}{\mathens{C}}

\newcommand{\T}{\mathens{T}}

\newcommand{\CP}{\C\mathrm{P}}

\newcommand{\RP}{\R\mathrm{P}}
\newcommand{\HP}{\mathens{H}\mathrm{P}}
\newcommand{\CaP}{\mathens{C}\mathrm{aP}^2}
\newcommand\sphere[1]{\mathens{S}^{#1}}

\DeclareMathOperator{\supp}{supp}

\newcommand{\ham}{\mathrm{Ham}}

\newcommand{\cont}{\mathrm{Cont}}
\newcommand{\conto}{\mathrm{Cont}_{0}}
\newcommand{\contoc}{\mathrm{Cont}_{0}^\mathrm{c}}

\newcommand{\tconto}{\widetilde{\mathrm{Cont}_{0}}}

\newcommand{\id}{\mathrm{id}}
\newcommand{\nlmas}\upmu

\newtheorem{thm}{Theorem}[section]
\newtheorem{lem}[thm]{Lemma}
\newtheorem{cor}[thm]{Corollary}
\newtheorem{prop}[thm]{Proposition}

\newtheorem{prop-def}[thm]{Definition-proposition}

\theoremstyle{definition}
\newtheorem{definition}[thm]{Definition}

\theoremstyle{remark}

\newtheorem{exs}[thm]{Examples}
\newtheorem{rem}[thm]{Remark}
\newtheorem{rems}[thm]{Remarks}

\newcommand\Leg{\mathcal{L}}
\newcommand\uLeg{\widetilde{\Leg}}
\newcommand\Gcont{\mathcal{G}}
\newcommand\uGcont{\widetilde{\Gcont}}
\newcommand\cleq{\preceq}
\newcommand\cgeq{\succeq}
\newcommand\spec{\mathrm{Spec}}
\newcommand\uLambda{\underline{\Lambda}}
\newcommand\dSCH{\mathrm{d}_\mathrm{SCH}}
\newcommand\dSH{\mathrm{d}_\mathrm{SH}}
\newcommand\dspec{\mathrm{d}_\mathrm{spec}}
\newcommand\dHosc{\mathrm{d}_\mathrm{H,osc}}
\newcommand\Nspec[1]{\left|#1\right|_\mathrm{spec}}
\newcommand\NHosc[1]{\left|#1\right|_\mathrm{osc}}
\newcommand\NSH[1]{\left|#1\right|_\mathrm{SH}}
\newcommand\NFPR[1]{\left|#1\right|_\mathrm{FPR}}

\newcommand\dFPR{\mathrm{d}_\mathrm{FPR}}

\newcommand\ldisc{\ell_\mathrm{disc}}
\newcommand\ddisc{\mathrm{d}_\mathrm{disc}}
\newcommand\losc{\nu_\mathrm{osc}}
\newcommand\dCSosc{\mathrm{d}_\mathrm{CS,osc}}

\DeclareMathOperator{\osc}{osc}

\DeclareFontFamily{U}{mathb}{\hyphenchar\font45}
\DeclareFontShape{U}{mathb}{m}{n}{
      <5> <6> <7> <8> <9> <10> gen * mathb
      <10.95> mathb10 <12> <14.4> <17.28> <20.74> <24.88> mathb12
}{}
\DeclareSymbolFont{mathb}{U}{mathb}{m}{n}
\DeclareMathSymbol{\cll}{3}{mathb}{"CE}
\DeclareRobustCommand{\cggP}{\text{\reflectbox{$\cll$}}}
\DeclareMathOperator{\cgg}{\cggP}

\makeatletter\let\@wraptoccontribs\wraptoccontribs\makeatother

\begin{document}

\title
{Spectral selectors and contact orderability}

\author[S. Allais]{Simon Allais}
\address{Simon Allais, IRMA,
Université de Strasbourg,
\newline\indent  7 rue Rene Descartes,
67084 Strasbourg, France}
\email{simon.allais@math.unistra.fr}
\urladdr{https://irma.math.unistra.fr/~allais/}

\author[P.-A. Arlove]{Pierre-Alexandre Arlove}
\address{P.-A. Arlove, Universit\'e de Strasbourg, IRMA UMR 7501, F-67000 Strasbourg, France}
\email{paarlove@unistra.fr}


\begin{abstract}
    We study the notion of orderability of isotopy classes
    of Legendrian submanifolds and their universal covers,
    with some weaker results concerning spaces of contactomorphisms.
    Our main result is that orderability is equivalent
    to the existence of spectral selectors analogous to
    the spectral invariants coming from Lagrangian Floer Homology.
    A direct application is the existence of Reeb chords between
    any closed Legendrian submanifolds of a same orderable isotopy class.
    Other applications concern the Sandon conjecture, the Arnold
    chord conjecture, Legendrian interlinking, the existence of time-functions and
    the study of metrics due to
    Hofer-Chekanov-Shelukhin, Colin-Sandon, Fraser-Polterovich-Rosen
    and Nakamura.
\end{abstract}

\maketitle

\section{Introduction}

\subsection{Historical background}
The abundance of interactions between Hofer geometry on the group of
Hamiltonian diffeomorphisms and symplectic geometry
\cite{hofer,lalondemcduff,geometricvariants,polterovich} led
Eliashberg and Polterovich in 2000 \cite{EP00} to introduce
the notion that will be later called \emph{orderability} on
the universal cover $\uGcont$ of the group of contactomorphisms isotopic to the
identity. To do so they define a bi-invariant
binary relation $\cleq$ on this group: $\varphi\cleq\psi$ if there
exists a non-negative contact Hamiltonian generating an isotopy
from $\varphi$ to $\psi$.
They show that for some contact manifolds this
binary relation turns out to be a partial order: we now
say that the group $\uGcont$ is orderable in this case (see Section \ref{se:contactorderability} for more details).
Some years later together with
Kim \cite{EKP} they show how orderability can be used to detect some
squeezing and non-squeezing phenomena.
They show in particular that there really is a dichotomy:
for some contact manifolds such as the standard contact sphere
$\uGcont$ is unorderable.

The study of orderability naturally extends to the group of
contactomorphisms isotopic to the identity $\Gcont$ and to
the isotopy class $\Leg$ of a closed Legendrian
submanifold $\Lambda_*$, i.e. the set of Legendrians that are Legendrian isotopic to $\Lambda_*$,  as well as its universal cover
$\uLeg$. The notion of orderability has been quite influential
in the study of contact topology: we refer to our paragraph
Examples~\ref{ex:orderable} in Section~\ref{se:contactorderability}
for an account on the study of the dichotomy
orderable/unorderable space throughout the last
two decades.
Nevertheless, rather little is known about the properties that
are specifically implied by orderability.
One of the most significant achievements in that direction is
the discovery of Albers-Fuchs-Merry that the unorderability
of $\Gcont(M,\xi)$ implies the Weinstein conjecture in $(M,\xi)$
\cite{AFM15}.
The goal of this article is to establish an equivalence between orderability
and the existence of (hopefully spectral) contact selectors.
\\

The notion of spectral selectors in contact geometry is a natural generalization
of spectral selectors defined in symplectic geometry.
Given a closed symplectic manifold $(M,\omega)$ satisfying some
mild hypothesis (\emph{e.g.} symplectic asphericality), one can
define spectral selectors on the space of Hamiltonian diffeomorphisms
$\ham(M,\omega)$ or its universal cover.
In this context a spectral selector is a continuous map
$c:\ham(M,\omega)\to\R$ associating to $\varphi$ a real number
$c(\varphi)$ belonging to its spectrum, which is the set of
action values of its fixed points.
Such selectors were first constructed by Viterbo for compactly
supported Hamiltonian diffeomorphisms of $\R^{2n}$
\cite{Vit} then generalized by the works of Schwarz \cite{schwarz}
and Oh \cite{Oh2005}.
Likewise, given an isotopy class of Lagrangian
submanifolds of $(M,\omega)$ satisfying some mild hypothesis,
one can associate spectral
selectors $\ell$ such that $\ell(\Lambda_1,\Lambda_0)$ is
in the spectrum of the couple of Lagrangian submanifolds
$(\Lambda_1,\Lambda_0)$ (which corresponds to symplectic areas
associated to couples of points of $\Lambda_0\cap\Lambda_1$)
\cite{Vit,Oh1997,Lec2008,LecZap2018}.
Among the spectral selectors $c$ or $\ell$ (a specific construction
being given), one can usually distinguish two specific selectors $c_-\leq c_+$
or $\ell_-\leq\ell_+$ (in the case of Hamiltonian Floer theory,
they correspond to the fundamental class and the class of a point)
which satisfy additional property, \emph{e.g.}
\begin{enumerate}[1.]
    \item (triangle inequality) $c_+(\varphi\psi) \leq
        c_+(\varphi) + c_+(\psi)$,
    \item (Poincaré duality) $c_+(\psi) = -c_-(\psi^{-1})$,
    \item (conjugation invariance) $c_\pm(\varphi\psi\varphi^{-1})
        = c_\pm(\psi)$,
    \item (non-degeneracy) $c_+(\psi)=c_-(\psi)$ implies
        $\psi = \id$.
\end{enumerate}
The number of applications of the existence of spectral selectors,
especially $c_+$, is quite large. Among other ones,
on the dynamical side, it has recently been used for the
study of the Hofer-Zehnder conjecture \cite{GG14,She19} or
the $C^\infty$-closing lemma \cite{CinSey2022},
it also has more topological applications:
the study of non-squeezing phenomena \cite{Vit},
the geometry of the group of Hamiltonian
diffeomorphisms \cite{BiaPol1994,Humiliere2008} or
$C^0$-symplectic geometry \cite{BHS2021}.   

This notion of spectral selectors has been extended to 
compactly supported contactomorphisms
of $\R^{2n}\times S^1$ isotopic to the identity by Sandon
in \cite{San11}. She discovered that the notion of translated
points was key in order to define the spectrum of
contactomorphisms (see Section~\ref{se:intro:spectral} below
for the definition).
These notions of spectral selectors were then applied in multiple contexts
in order to address numerous questions of contact topology:
study of bi-invariant metrics on the group of compactly
supported contactomorphisms isotopic to identity \cite{Sandon2010,discriminante},
contact non-squeezing phenomena \cite{San11,camel}, orderability \cite{San2011}.
Albers-Fuchs-Merry use the Rabinowitz Floer theory to define contact
selectors in a class of Liouville fillable contact manifolds
\cite{AFM15,albers} (see in Section~\ref{se:existenceChords} some applications of
their work).
Our main result is that orderability is the necessary and sufficient
condition in order to define contact analogues of selectors
$c_\pm$ and $\ell_\pm$ in full generality
(although our $c_\pm$'s are not spectral \emph{a priori})
whose signs only are invariant by contactomorphisms
(\emph{cf.} paragraph just above Corollary~\ref{cor:orderEquivSelector}).

\subsection{Conventions}

Whenever no precision regarding regularity is given, maps and manifolds considered
are $C^\infty$-smooth.
Every contact manifold $(M,\xi)$ considered are assumed to be connected and cooriented
(\emph{i.e.} $TM/\xi$ is an oriented line bundle).
We write that a contact form $\alpha$ is supporting $\xi$ if
$\ker\alpha=\xi$ and $\alpha$ respects the coorientation.
A contactomorphism of $(M,\xi)$ will always preserve the coorientation.
Given a contact manifold
$(M,\xi)$ and a closed Legendrian submanifold $\Lambda_*\subset M$,
one denotes $\Leg(\Lambda_*)$ the isotopy class of $\Lambda_*$.  
One denotes $\Gcont(M)$
the group $\contoc(M,\xi)$ of compactly supported
contactomorphisms of $(M,\xi)$
isotopic to $\id$ through compactly supported
contactomorphisms while $\uGcont(M)$ will denote its universal cover. For more details about the topology considered on these spaces we refer the reader to Section \ref{se:topologyLeg}.

Usually, we will only write $\Leg$ for $\Leg(\Lambda_*)$, $\uLeg$ for
$\uLeg(\Lambda_*)$ etc. implicitly meaning that $\Lambda_*$ and $M$ are fixed.
By a slight abuse of notation:
\begin{enumerate}[ 1.]
    \item $\id\in\uGcont$ will denote
the class of the constant isotopy $s\mapsto \id$.
Given a contact form $\alpha$ supporting $\xi$ and $t\in\R$,
let $\phi^\alpha_t\in\cont(M,\xi)$ denote the $\alpha$-Reeb flow at
time $t\in\R$ (see \emph{e.g.} Section~\ref{se:LegendrianIsotopies}); 
    \item when it is understood that $\phi^\alpha_t$ must be an element of
the universal cover of the identity component of $\cont(M,\xi)$
(\emph{e.g.} when acting on $\uGcont$ or $\uLeg$),
it will denote the isotopy $[0,1]\to\conto(M,\xi)$, $s\mapsto \phi^\alpha_{st}$;
    \item the cover maps
$\uGcont\to\Gcont$ and  $\uLeg\to\Leg$ will both be denoted $\Pi$.
\end{enumerate}

On $O$ being either $\Leg$, $\uLeg$, $\Gcont$ or $\uGcont$,
we write $x\cleq y$, or equivalently $y\cgeq x$, if there exists
a non-negative isotopy in $O$ from $x$ to $y$ (see Section~\ref{se:prelim} for more details).
$O$ is called orderable if and only if $\cleq$ defines a partial
order (see Section~\ref{se:contactorderability} for details).
In particular, the orderability of $\Leg$ (resp. $\Gcont$) implies
the orderability of $\uLeg$ (resp. $\uGcont$).
We refer to Examples~\ref{ex:orderable} for a list of known orderable
spaces.

\subsection{Order spectral selectors}
\label{se:intro:spectral}
Let $(M,\xi)$ be a closed cooriented contact manifold.
For any contact form $\alpha$ supporting $\xi$,
let us define $c_-^\alpha$ and $c_+^\alpha$
as the maps $\Gcont\to\R\cup\{\pm\infty\}$ (resp. $\uGcont\to\R\cup\{\pm\infty\}$)
defined by
\begin{equation}\label{eq:c}
    c_-^\alpha(\psi) := \sup\{ t\in\R \ |\ 
    \psi \cgeq \phi_t^\alpha \} \quad\text{and}\quad
    c_+^\alpha(\psi) := \inf\{ t\in\R \ |\ 
    \psi \cleq \phi_t^\alpha \},
\end{equation}
for $\psi\in\Gcont$ (resp. $\psi\in\uGcont$).
These maps are highly inspired by constructions of
Fraser-Polterovich-Rosen \cite{FPR} and were already considered by the second
author in his PhD-thesis \cite{ArlovePhD} and very recently by
Nakamura \cite{Nakamura2023} from a metrical point of view.

Let us now withdraw the closeness
assumption on $M$ and let $\Leg = \Leg(\Lambda_*)$ for some closed Legendrian
$\Lambda_*\subset M$.   For any complete contact form $\alpha$
supporting $\xi$ (\emph{i.e.} a contact form the Reeb vector
field of which is complete),
let us define $\ell_-^\alpha$ and $\ell_+^\alpha$ as the maps
$\Leg\times\Leg\to\R\cup\{\pm\infty\}$ (resp.
$\uLeg\times\uLeg\to\R\cup\{\pm\infty\}$) given by
\begin{equation*}
    \ell_-^\alpha(\Lambda_1,\Lambda_0) := \sup\{ t\in\R \ |\ 
    \Lambda_1 \cgeq \phi_t^\alpha \Lambda_0 \} \quad\text{and}\quad
    \ell_+^\alpha(\Lambda_1,\Lambda_0) := \inf\{ t\in\R \ |\ 
    \Lambda_1 \cleq \phi_t^\alpha \Lambda_0 \},
\end{equation*}
for $\Lambda_0,\Lambda_1\in\Leg$ (resp. $\uLeg$).

For $\psi\in\Gcont$ and $\alpha$ supporting $\xi$,
the $\alpha$-spectrum of $\psi$ is the set of time-shifts of its $\alpha$-translated
points that is
\begin{equation*}
    \spec^\alpha(\psi) := \{ t\in\R \ | \ \exists p \in M,
        (\psi^*\alpha)_p = \alpha_p \text{ and }
    \psi(x) = \phi_t^\alpha(x) \}.
\end{equation*}
For $\psi\in\uGcont$, the $\alpha$-spectrum is defined as the
$\alpha$-spectrum of $\Pi\psi$.
For $\Lambda_0,\Lambda_1\in\Leg$, the $\alpha$-spectrum of $(\Lambda_1,\Lambda_0)$ is
the set of (positive and non-positive) lengths of $\alpha$-Reeb chords
joining $\Lambda_0$ to $\Lambda_1$ that is
\begin{equation*}
    \spec^\alpha(\Lambda_1,\Lambda_0) := \{ t\in\R \ | \ \Lambda_1\cap\phi_t^\alpha \Lambda_0
    \neq\emptyset \}.
\end{equation*}
For $\Lambda_0,\Lambda_1\in\uLeg$, the $\alpha$-spectrum of $(\Lambda_1,\Lambda_0)$
is defined as the $\alpha$-spectrum of $(\Pi\Lambda_1,\Pi\Lambda_0)$.
The latter notion of spectrum can be seen as a generalization of the former.
Indeed, if $\psi\in\cont(M,\ker\alpha)$, its contact graph
\begin{equation*}
    \mathrm{gr}^\alpha(\psi):= \{ (x,\psi(x),g(x)) \ |\ x\in M\},
    \text{ where } \psi^*\alpha = e^g\alpha,
\end{equation*}
is a Legendrian submanifold of the contact manifold $M\times M\times\R$ endowed
with the contact form $\tilde{\alpha}:=\alpha_2-e^\theta\alpha_1$ defined in
Example~\ref{ex:orderable}.\ref{ex:it:SH}.  Then,
\begin{equation*}
    \spec^\alpha(\psi) = \spec^{\tilde{\alpha}}(\mathrm{gr}^\alpha(\psi),
    \mathrm{gr}^\alpha(\id)).
\end{equation*}

\begin{thm}[Legendrian order spectral selectors]\label{thm:LegSS}
    Let $\Lambda_*$ be a closed Legendrian submanifold of $(M,\xi)$
    such that $\Leg(\Lambda_*)$ (resp. $\uLeg(\Lambda_*)$) is orderable
    and let $\alpha$ be a complete contact form supporting $\xi$.
    The selectors $\ell_\pm^\alpha$ are real-valued and satisfy the
    following properties for every $\Lambda_0,\Lambda_1,\Lambda_2\in\Leg$ (resp. $\uLeg$),
    \begin{enumerate}[ 1.]
        \item (normalization) $\ell_\pm^\alpha(\Lambda_0,\Lambda_0) = 0$
            and $\ell_\pm^\alpha(\phi^\alpha_t \Lambda_1,\Lambda_0)
            = t + \ell_\pm^\alpha(\Lambda_1,\Lambda_0)$, $\forall t\in\R$,
        \item (monotonicity) $\Lambda_2\cleq \Lambda_1$ implies
            $\ell_\pm^\alpha(\Lambda_2,\Lambda_0) \leq \ell_\pm^\alpha(\Lambda_1,\Lambda_0)$,
        \item (triangle inequalities) $\ell_+^\alpha(\Lambda_2,\Lambda_0) \leq
            \ell_+^\alpha(\Lambda_2,\Lambda_1) + \ell_+^\alpha(\Lambda_1,\Lambda_0)$
            and $\ell_-^\alpha(\Lambda_2,\Lambda_0) \geq \ell_-^\alpha(\Lambda_2,\Lambda_1)
            + \ell_-^\alpha(\Lambda_1,\Lambda_0)$,
        \item (Poincaré duality) $\ell_+^\alpha(\Lambda_1,\Lambda_0) =
            - \ell_-^\alpha(\Lambda_0,\Lambda_1)$,
        \item (compatibility) $\ell_\pm^\alpha(\varphi(\Lambda_1),\varphi(\Lambda_0))
            = \ell_\pm^{\varphi^*\alpha}(\Lambda_1,\Lambda_0)$ 
            
    \noindent        
(resp. $\ell^\alpha_\pm(\varphi(\Lambda_1),\varphi(\Lambda_2))=\ell_\pm^{\Pi(\varphi)^*\alpha}(\Lambda_1,\Lambda_0)$),
            for every $\varphi\in\cont(M,\xi)$
            (resp. $\varphi\in\widetilde{\cont}(M,\xi)$),
        \item (non-degeneracy) $\ell_+^\alpha(\Lambda_1,\Lambda_0) = \ell_-^\alpha(\Lambda_1,\Lambda_0) = t$
            for some $t\in\R$ implies $\Lambda_1 = \phi_t^\alpha \Lambda_0$
            (resp. 
            $\Pi \Lambda_1 = \phi_t^\alpha\Pi \Lambda_0\in\Leg$).
        \item (spectrality) $\ell_\pm^\alpha(\Lambda_1,\Lambda_0) \in \spec^\alpha(\Lambda_1,\Lambda_0)$.
    \end{enumerate}
\end{thm}

\begin{thm}[Contactomorphism order selectors]
    \label{thm:contSS}
    Let $(M,\xi)$ be a closed contact manifold such that
    $\Gcont$ (resp. $\uGcont$) is orderable and let
    $\alpha$ be a contact form supporting $\xi$.
    The selectors $c_\pm^\alpha$ are real-valued and satisfy
    the following properties for $\varphi,\psi\in\Gcont$
    (resp. $\varphi,\psi\in\uGcont$):
    \begin{enumerate}[ 1.]
        \item (normalization) $c_\pm^\alpha(\id) = 0$ and
            $c_\pm^\alpha(\phi_t^\alpha\psi) = t + c_\pm^\alpha(\psi)$,
            $\forall t\in\R$,
        \item (monotonicity) $\varphi\cleq\psi$ implies $c_\pm^\alpha(\varphi)\leq
            c_\pm^\alpha(\psi)$,
        \item (triangle inequalities) $c_+^\alpha(\varphi\psi) \leq c_+^\alpha(\varphi)
            + c_+^\alpha(\psi)$ and
            $c_-^\alpha(\varphi\psi) \geq c_-^\alpha(\varphi) + c_-^\alpha(\psi)$,
        \item (Poincaré duality) $c_+^\alpha(\psi) = - c_-^\alpha(\psi^{-1})$,
        \item (compatibility) $c_\pm^\alpha(\varphi\psi\varphi^{-1})
            = c_\pm^{\varphi^*\alpha}(\psi)$ (resp. $c_\pm^\alpha(\varphi\psi\varphi^{-1})=c_\pm^{\Pi(\varphi)^*\alpha}(\psi)$), which extends to every
            $\varphi\in\cont(M,\xi)$ (resp. $\varphi\in\widetilde{\cont}(M,\xi)$),
        \item (non-degeneracy) $c_-^\alpha(\psi) = c_+^\alpha(\psi) = t$ 
            for some $t\in\R$ implies $\psi = \phi^\alpha_t$
            (resp. $\Pi\psi = \phi^\alpha_t\in\Gcont$).
    \end{enumerate}
\end{thm}

The major results in these statements concern the non-degeneracy property and the
spectrality. The non-degeneracy property has also been recently proven
by Nakamura with a different approach \cite{Nakamura2023}.

We emphasize the fact that we do not know that the $c_\pm^\alpha(\psi)$'s
do actually select spectral values of $\psi$.
Nevertheless, in the few examples we are able to compute explicitly $c^\alpha_\pm$ (for example see Remark \ref{rem:spectraliteconj} and \cite[Corollary 5.1]{allais2024c1localflatnessgeodesicslegendrian}), these maps do take spectral values and therefore strongly suggest that they should be spectral in general.

The following direct corollary states that orderability is equivalent to the
existence of selectors in a rather weak sense
(a similar statement for spaces $\Leg$ and $\uLeg$ also holds). In contrast with their symplectic counterparts, these selectors
are not actually invariant: identities
$\ell(\varphi\Lambda_1,\varphi\Lambda_0) = \ell(\Lambda_1,\Lambda_0)$
and
$c(\varphi\psi\varphi^{-1}) = c(\psi)$ are not verified
for every contactomorphism $\varphi$ (resp. lift in the universal cover),
$\Lambda_0,\Lambda_1$ and $\psi$ being fixed.
In that sense, the $\alpha$-selectors are only invariant under
$\alpha$-strict contactomorphisms:
$\ell^\alpha_\pm(\varphi\Lambda_1,\varphi\Lambda_0)
= \ell^\alpha_\pm(\Lambda_1,\Lambda_0)$ 
and $c^\alpha_\pm(\varphi\psi\varphi^{-1})=c^\alpha_\pm(\psi)$ when
$\varphi^*\alpha = \alpha$.
Of course, this is compatible with the fact that the $\alpha$-spectrum
is only invariant under strict contactomorphisms:
as in general $\varphi^{-1}\phi^\alpha_t\varphi = \phi^{\varphi^*\alpha}_t$ for
all $t\in\R$, one has
\begin{equation*}
    \spec^\alpha(\varphi\Lambda_1,\varphi\Lambda_0)
    = \spec^{\varphi^*\alpha}(\Lambda_1,\Lambda_0)
    \text{ and }
    \spec^\alpha(\varphi\psi\varphi^{-1})
    = \spec^{\varphi^*\alpha}(\psi).
\end{equation*}
Nonetheless, the sign of these selectors is invariant
(see Lemmata~\ref{lem:sign} and \ref{lem:cont-sign}).

\begin{cor}\label{cor:orderEquivSelector}
    Let $(M,\xi)$ be a closed cooriented contact manifold.
    The space $\Gcont$ (resp. $\uGcont$) is orderable
    if and only if there exists
    a non-decreasing map $c:(\Gcont,\cleq)\to(\R,\leq)$
    (resp. $(\uGcont,\cleq)\to(\R,\leq)$) such that
    \begin{enumerate}[ 1.]
        \item $c(\id) = 0$,
        \item$\psi\cgeq\id$ and $\psi\neq\id$ implies $c(\psi)>0$.
    \end{enumerate}
\end{cor}

\begin{proof}[Proof of Corollary~\ref{cor:orderEquivSelector}]
    The direct implication is a consequence of Theorem~\ref{thm:contSS}
    (\emph{cf.} Corollary~\ref{cor:cont-positivity} below).
    Conversely, assuming the existence of $c$, if $\varphi,\psi\in\Gcont$
    (resp. $\uGcont$) are such that $\varphi\cleq\psi$ and
    $\psi\cleq\varphi$ while $\varphi\neq\psi$,
    then $c(\varphi\psi^{-1})>0$ by the positivity property of $c$
    while $\varphi\psi^{-1}\cleq\id$ implies $c(\varphi\psi^{-1})\leq c(\id)=0$
    by monotonicity of $c$, a contradiction.
\end{proof}

When $\Leg$ (resp. $\uLeg$) is orderable and $\Lambda_0\in\Leg$
(resp. $\uLeg$),
any non-decreasing map $\ell(\cdot,\Lambda_0):\Leg\to\R$
(resp. $\uLeg\to\R$)
that is normalized with respect to the contact form $\alpha$ in the sense of
Theorem~\ref{thm:LegSS} satisfies
\begin{equation*}
    \ell_-^\alpha(\cdot,\Lambda_0) \leq \ell(\cdot,\Lambda_0) \leq 
    \ell_+^\alpha(\cdot,\Lambda_0),
\end{equation*}
so that $\ell_-^\alpha$
can be thought of as the
minimal $\alpha$-spectral selector and $\ell_+^\alpha$ can be thought
as the maximal one.
A similar statement can be asserted for $c_\pm^\alpha$.

\begin{rem}[Removing the closeness assumption on $M$]
    When studying $\Gcont$ (resp. $\uGcont$), we have added the additional
    hypothesis of closeness for the contact space $M$.
    Following Fraser-Polterovich-Rosen \cite{FPR}, one can remove this hypothesis
    by asking orderability not only for the group $\Gcont$ (resp. $\uGcont$)
    but for the group $\Gcont_\alpha$ (resp. $\uGcont_\alpha$) generated by
    $\Gcont$ (resp. $\uGcont$) and the
    elements of the Reeb flow $(\phi^\alpha_t)$, once fixed a complete supporting
    contact form $\alpha$. This way the relation $\psi\cleq\phi_t^\alpha$
    makes sense for $\psi\in\Gcont$ (resp. $\uGcont$) and $t\in\R$,
    and Theorem~\ref{thm:contSS} still holds.
\end{rem}

\subsection{Applications to the existence of Reeb chords
and translated points}
\label{se:existenceChords}

A direct consequence of the existence of spectral selectors in orderable
Legendrian isotopy classes, is the non-triviality of the spectrum.

\begin{cor}[Existence of Reeb chords]\label{cor:intro:existence}
    Let $\Lambda_*$ be a closed Legendrian submanifold of
    a cooriented contact manifold $(M,\xi)$ such that
    $\uLeg(\Lambda_*)$ is orderable.
    Then any two Legendrian submanifolds $\Lambda_0,\Lambda_1\in\Leg(\Lambda_*)$
    are joined by two distinct (positive or non-positive)
    $\alpha$-Reeb chords, given any complete $\alpha$ supporting $\xi$.\\
    In particular, if $(M,\xi)$ is a closed contact manifold such that
    $\uLeg(\Delta\times \{ 0\})$ as defined in
    Example~\ref{ex:orderable}.\ref{ex:it:SH} is orderable,
    any contactomorphism
    of $M$ isotopic to $\id$ has an $\alpha$-translated point
    (and at least two distinct couples $(p,t)$ where $p$ is a translated
    point of time-shift $t$), given any $\alpha$ supporting $\xi$.
\end{cor}

The existence of Reeb chords between two isotopic Legendrian
submanifolds has been widely studied.
When Legendrian submanifolds are isotopic to $\RP^n$ in the standard
$\RP^{2n+1}$ (\emph{cf.} Example~\ref{ex:orderable}.\ref{ex:it:RP}),
the non-linear Maslov index developed by Givental \cite{Giv90}
directly implies that the number of chords whose images are
not overlapping each other is at least $n+1$
for the standard contact form.
When Legendrian submanifolds are Legendrian isotopic to the zero-section of
the standard $J^1 N$ where $N$ is closed (\emph{cf}
Example~\ref{ex:orderable}.\ref{ex:it:JN}), Chekanov gave
lower bounds on the number of chords based on Morse theory
(sum of the Betti numbers plus twice torsion numbers in the
generic case, cuplength in the general case) for the standard
contact form \cite{Chekanov1996}.
With the exception of these two cases and the question of translated
points discussed below, 
most of the known results concern the case of a \emph{small} Legendrian
isotopy in various senses (\emph{e.g.} regarding $C^1$-metric,
$C^0$-metric or Hofer type metrics),
we refer to \cite{RizSul2020,Li2021} and references therein
where bounds similar to Chekanov's are obtained.
Of course, one can easily construct examples of isotopic Legendrian submanifolds
without any Reeb chord joining them (\emph{e.g.} a Legendrian circle
in $J^1\R$ pushed horizontally).
In contrast with these local studies, Corollary~\ref{cor:intro:existence}
asserts that orderability of $\uLeg(\Lambda_*)$ is enough
for the existence of two chords in between any two Legendrian submanifolds
isotopic to $\Lambda_*$, no matter how far they are from each other.
In addition to this global statement,
it is remarkable that one gets the existence of two distinct
Reeb chords without any transversality assumption.
Indeed, sharper estimates on the number of Reeb chords involving Morse type estimates
are known in the aforementioned local studies but most of them require
additional transversality hypothesis.
Without these hypothesis, it is not clear that
one could keep two distinct Reeb chords: when passing from a generic case to the
general one, one usually looses the multiplicity (\emph{e.g.}
the classical Lefschetz fixed-point theorem or
more recently the critical point theory of closed 1-forms
\cite{Farber2002}).

The existence of translated points is an important question that
had been asked by Sandon in \cite{San13}.
Originally, Sandon asked whether the number of translated points
of any element of $\Gcont(M,\xi)$ was at least the minimal number
of critical points of a map $M\to\R$, by analogy with the Arnold conjecture
bounding from below the number of fixed points a Hamiltonian diffeomorphism
can have. This original statement was proven in the case of projective
spaces and lens spaces endowed with their standard contact form
\cite{San13,lens}. Lower bound related to Morse theory have been discovered
in multiple other cases including the whole class of non-degenerate
contactomorphisms of hypertight
contact manifolds (\emph{i.e.} admitting a Reeb flow without any
contractible closed orbit) \cite{AFM15,MN18,Allais2022UnitTangent}.
However, Cant proved recently that some elements of $\Gcont(\sphere{2n+1})$
do not admit any translated point for the standard contact form
of the sphere of dimension $2n+1\geq 3$, answering negatively
to Sandon's question.
As the space $\uGcont(\sphere{2n+1})$ is unorderable, this result
is compatible with our conjecture that the selectors $c_\pm^\alpha$
are spectral (which implies the existence of translated points
when $\uGcont$ is orderable).
Concerning the lower bound, the original conjecture of Sandon might
be too optimistic but variants exist:
see the introduction of \cite{Allais2022UnitTangent}.

Corollary~\ref{cor:intro:existence} can in fact be generalized
to a statement about positive Legendrian paths.

\begin{thm}[Intersection of a positive Legendrian path]\label{thm:existence}
    Let $\Lambda_*$ be a closed Legendrian submanifold of
    a cooriented contact manifold such that
    $\uLeg(\Lambda_*)$ is orderable.
    Given any uniformly positive isotopy $\uLambda$ in $\Leg(\Lambda_*)$
    (\emph{c.f.} the beginning of Section~\ref{se:spectralSelectors}),
    any Legendrian submanifold $\Lambda\in\Leg(\Lambda_*)$
    intersect $\uLambda_t$ for at least two distinct $t\in\R$
    unless $\Lambda=\uLambda_t$ for some $t\in\R$.
\end{thm}

Another consequence of our spectral selectors relating to Reeb chords
is the following theorem addressing the Arnold chord conjecture
(see a discussion of the conjecture below).

\begin{thm}[About the Arnold chord conjecture]\label{thm:ArnoldChord}
    Let $\Lambda_*\subset (M,\xi)$ be a Legendrian submanifold
    such that $\Leg(\Lambda_*)$ is unorderable and
    $\uLeg(\Lambda_*)$ is orderable.
    Then for every Legendrian submanifold $\Lambda\in\Leg(\Lambda_*)$
    and every complete contact form $\alpha$ supporting $\xi$, there
    exist two distinct non-constant Reeb chords $\gamma_i:[0,1]\to M$,
    $i\in\{1,2\}$, such that $\gamma_i(\{0,1\})\subset\Lambda$.
\end{thm}

Having a periodic Reeb flow for some contact form supporting
$\xi$ is a sufficient condition to ensure the unorderability of
$\Leg$ (in this case Theorem~\ref{thm:ArnoldChord} is relevant
when applied to other contact forms supporting $\xi$).
This can be seen as a Legendrian analogue of the theorem of
Albers-Fuchs-Merry addressing the Weinstein conjecture:
every closed contact manifold such that $\Gcont$ is unorderable
admits a closed Reeb chord for any supporting contact form
\cite{AFM15}.
In our setting, we still ask for the orderability
of the universal cover of $\Leg$.
Nonetheless, let us remark that this theorem can be applied to
non closed contact manifolds, only the Legendrian submanifolds need
to be closed, and the statement gives at least two distinct Reeb chords.
Therefore, it implies the Weinstein conjecture on $(M,\xi)$ in the case
$\Lambda_* := \Delta\times\{ 0\} \subset M\times M\times\R$ satisfies
the hypothesis of Theorem~\ref{thm:ArnoldChord} for the contact structure of
Example~\ref{ex:orderable}.\ref{ex:it:SH};
this set of $(M,\xi)$ could possibly include cases not treated by the
aforementioned theorem of Albers-Fuchs-Merry.

In its original setting, the Arnold chord conjecture asked about
the existence of a non-trivial Reeb chord the endpoints of which
lie on a same Legendrian circle of the standard contact 3-sphere, given
any Legendrian circle \cite[\S 8]{Arnold1986}.
In this setting Mohnke proved a generalization to closed
contact manifolds arising as the boundary of subcritical Stein
manifolds
(Legendrian circles are replaced by closed Legendrian submanifolds)
\cite{Mohnke2001}.
Hutching and Taubes extended it to any closed contact 3-manifolds
\cite{HutTau2011}.
Very recent works still discuss generalized forms of the original conjecture:
Chantraine proved a version for a specific kind of fillable contact manifolds
with Lagrangian slices taking the place of Legendrian submanifolds
\cite{Chantraine2022}, the case of Legendrian submanifolds lying in
$P\times\R$, for a Liouville manifold $P$, has also received
a lot of attention (see \cite{Li2021} and references therein).

Soon before releasing this article,
Leonid Polterovich brought to our attention another potential
class of applications: the study of robust Legendrian interlinking
that had been introduced by Entov-Polterovich \cite{EntPol2017}.
It turns out that our spectral selectors do bring
new insights to this notion.
We refer directly to Section~\ref{se:interlinked} for more details.

\subsection{Applications to the metric structures of
$\Leg$, $\uLeg$, $\Gcont$ and $\uGcont$}

Since the discovery by Hofer of a bi-invariant metric on the
group of Hamiltonian diffeomorphisms \cite{hofer},
multiple attempts to obtain a contact counterpart where made.
On the one hand, the obvious generalization of the Hofer metric,
often referred to as the Shelukhin-Hofer metric, is not bi-invariant.
On the other hand, one can define interesting bi-invariant metrics on $\Gcont$
and $\uGcont$ but they are discrete, as was discovered by Colin-Sandon
\cite{discriminante} and
Fraser-Polterovich-Rosen \cite{FPR}.
Our spectral selectors allow us to define a contact analogue of the
spectral metric discovered by Viterbo in \cite{Vit}.
We use this new metric to study both Hofer type contact metrics
and the discrete bi-invariant metrics of Colin-Sandon and Fraser-Polterovich-Rosen. In particular, this new metric allows us to study questions of non-degeneracy, unboundedness, geodesics and equivalence of the different mentionned metrics.
We directly refer to Section~\ref{se:metrics} for details.

We point out that the spectral metric has also been introduced by
the very recent work of Nakamura \cite{Nakamura2023} (without reference to
spectrality) in order to show the metrizability of the interval topology
and to study the Hofer type contact metrics.
Our results and his do overlap on these two subjects.

A last metric application of our spectral selectors concerns the recent
Lorentz-Finsler structure introduced by
Abbondandolo-Benedetti-Polterovich \cite{ABP2022}.
A natural question arising from their study was the existence of
a so-called \emph{time function}, a natural object coming from
the theory of relativity \cite[Question~K.1]{ABP2022}.
The following theorem is a positive answer in a broader setting to their
question.

\begin{thm}[Existence of time functions]\label{thm:intro:timefct}
    Let $O$ be either $\Leg,\uLeg,\Gcont$ or $\uGcont$ associated
    with a cooriented contact manifold $(M,\xi)$
    (which is closed if $O=\Gcont$ or $\uGcont$).
    If $O$ is orderable and $\alpha$ is supporting $\xi$, then 
    there exists a continuous non-decreasing map
    $\tau^\alpha:(O,\cleq)\to(\R,\leq)$ such that $x\cleq y$ and $x\neq y$
    implies $\tau^\alpha(x)<\tau^\alpha(y)$ and
    $\tau^\alpha(\phi_t^\alpha x) = t + \tau^\alpha(x)$ for all $t\in\R$ and
    $x,y\in O$.
\end{thm}

The construction is highly non-canonical and one cannot expect much more
natural properties satisfied by our time functions.
Nevertheless, it is natural to ask whether there exists a time function
on the orderable group $\Gcont$ or $\uGcont$
that can be conjugation invariant.
We show that such a map does not exist (Theorem~\ref{thm:noInvTimeFunction}).

\subsection*{Organization of the paper}
In Section~\ref{se:prelim}, we provide the background on
contact geometry needed throughout the article.
In Section~\ref{se:spectralSelectors}, we study the order spectral
selectors defined in the introduction and prove
Theorems~\ref{thm:LegSS} and \ref{thm:contSS} as well as
the results introduced in Section~\ref{se:existenceChords}.
In Section~\ref{se:metrics}, we define the spectral metric and use
it to study Hofer type metrics, Colin-Sandon metrics and the
Fraser-Polterovich-Rosen metric.
In Section~\ref{se:LorentzFinsler}, we show Theorem~\ref{thm:intro:timefct}.

\subsection*{Acknowledgement} We are grateful to the organizers of the
\emph{School on conformal symplectic dynamics and related fields}, this school
was the starting point of our collaboration on the subject.  We would like to
thank Margherita Sandon for her constant support.  The first author also thanks
Baptiste Chantraine, Erman \c Cineli, Vincent Humilière and Egor Shelukhin for
their support and fruitful discussions. The second author also thanks Alberto
Abbondandolo, Souheib Allout, Jakob Hedicke and Stefan Nemirovski for
interesting and stimulating discussions.

The first author was supported by the postdoctoral fellowship of the Fondation
Sciences Math\'ematiques de Paris, the postdoctoral fellowship of the
Interdisciplinary Thematic Institute IRMIA++ and the ANR AAPG 2021 PRC CoSyDy:
Conformally symplectic dynamics, beyond symplectic dynamics, ANR-CE40-0014. The
second author is partially supported by the Deutsche Forschungsgemeinschaft
under the Collaborative Research Center SFB/TRR 191 - 281071066 (Symplectic
Structures in Geometry, Algebra and Dynamics).

\section{Preliminaries}
\label{se:prelim}

\subsection{Legendrian and contact isotopies and their Hamiltonian maps}
\label{se:LegendrianIsotopies}

Throughout this section $I\subset\R$ will always be an interval
containing $0$.
Let $(M,\xi)$ be a cooriented contact manifold. A contact form
$\alpha$ supporting $\xi$ (\emph{i.e.} $\ker\alpha=\xi$) is called
\emph{complete} if its Reeb vector field $R^\alpha$ induces a
complete Reeb flow.
We recall that the Reeb vector field $R^\alpha$ is uniquely defined
by the equations $\alpha(R^\alpha)\equiv 1$ and $\iota_{R^\alpha}\ud\alpha=0$.
Let us fix a complete contact form $\alpha$ supporting $\xi$.

Let us recall that a contact isotopy $(\varphi_t)_{t\in I}$ is a family of contactomorphisms such that the map $I\times M\to M, (t,x)\mapsto \varphi_t(x)$ is smooth. For any contact isotopy $(\varphi_t)$ such that $\varphi_0=\id$,
there exists a unique smooth map $H:I\times M\to\R$ related to
the time-dependent contact vector field $(X_t)_{t\in I}$ generating the isotopy, i.e. $\frac{d}{dt}\varphi_t=X_t\circ\varphi_t$, by
\begin{equation}\label{eq:contactHamilton}
    \begin{cases}
        \alpha(X_t) = H_t, \\
        \iota_{X_t}\ud H_t = (\ud H_t\cdot R^\alpha)\alpha - \ud H_t,
    \end{cases}
\end{equation}
where $R^\alpha$ denotes the Reeb vector field of $\alpha$
and $H_t:M\to\R$ is induced by the restriction of $H$ to $\{ t\}\times M$
(see \emph{e.g.} \cite[\S 2.3]{Gei}). 
 \begin{definition}
     Let $(\varphi_t)_{t\in I}$ be a contact isotopy such that $\varphi_0=\id$. We call the map $H : I\times M\to\R$ satisfying \eqref{eq:contactHamilton} the $\alpha$-contact Hamiltonian map of $(\varphi_t)$.

 \end{definition}

Conversely, to any smooth map $H:I\times M\to\R$
is associated a unique time-dependent contact vector field $(X_t)$
by the system (\ref{eq:contactHamilton}).
When the contact vector field is complete (\emph{e.g.} when the
differential of $H$ is compactly supported), we say that $H$
generates the contact isotopy obtained by integrating $(X_t)$.\\

The notion of Hamiltonian map naturally extends to Legendrian isotopies.
Throughout the paper, we only consider closed Legendrian submanifolds.
A Legendrian isotopy $(\Lambda_t)_{t\in I}$ is a family of
Legendrian submanifolds $\Lambda_t\subset M$, such that there
exists a smooth map $j:I\times\Lambda_0\to M$
whose restriction $j_t$ to $\{ t\}\times\Lambda_0$ induces a
diffeomorphism $\Lambda_0\simeq\Lambda_t$ for every $t\in I$.
Given such an isotopy $(\Lambda_t)$ and such a map $j$,
let us define the family of maps $H_t:\Lambda_t\to\R$, $t\in I$,
by
\begin{equation}\label{eq:LegHam}
    H_t\circ j(t,x) := \alpha\left(\frac{\partial j}{\partial t}(t,x)\right),
    \quad \forall (t,x)\in I\times\Lambda_0.
\end{equation}
As $j_t^*\alpha = 0$ for all $t\in I$, one can deduce that $(H_t)$
only depends on the isotopy $(\Lambda_t)$ and not the specific
choice of parametrization $j$.

\begin{definition}
Let $(\Lambda_t)_{t\in I}$ be a Legendrian isotopy. The family $(H_t : \Lambda_t\to\R)_{t\in I}$ satisfying \eqref{eq:LegHam} is called the $\alpha$-contact Hamiltonian map
of $(\Lambda_t)$.
\end{definition}

The contact Hamiltonian map of a Legendrian isotopy $(\Lambda_t)_{t\in I}$ can be seen as a smooth
map $H:N\to\R$ defined on the submanifold $N := \bigcup_t \{t\}\times\Lambda_t$
of $I\times M$.
If $\Lambda_t=\varphi_t(\Lambda_0)$ for some contact isotopy
$(\varphi_t)$ starting at the identity whose $\alpha$-Hamiltonian map is $K : I\times M\to\R$, the Hamiltonian of $(\Lambda_t)$
is the restriction of $K$ to $N$.

We will use the following unparametrized version of the
Legendrian isotopy extension theorem.

\begin{lem}\label{lem:isotopyExtension}
    Let $I\subset\R$ be an interval containing $0$ and
    let $(\Lambda_t)_{t\in I}$ be a Legendrian
    isotopy on $(M,\ker\alpha)$,
    the $\alpha$-contact Hamiltonian map of which is $H:\bigcup_t
    \{t\}\times\Lambda_t\to\R$.
    Let $K:I\times M\to\R$ be an $\alpha$-contact Hamiltonian smoothly
    extending $H$ and generating a contact isotopy $(\varphi_t)$,
    then $\varphi_t(\Lambda_0) = \Lambda_t$ for all $t\in I$.
\end{lem}

Before proving Lemma~\ref{lem:isotopyExtension},
let us recall the link between Legendrian isotopies and the Hamilton-Jacobi
equation. Let $(\Lambda_t)$ be a Legendrian isotopy. According to the
Legendrian Weinstein neighborhood theorem, there exists an open subset $W\subset M$
containing $\Lambda_0$ that is contactomorphic to a neighborhood
of the zero-section of the 1-jet space $J^1\Lambda_0:= T^*\Lambda_0\times\R$,
through this identification $\Lambda_0$ is sent to the zero-section.
Moreover, our chosen contact form $\alpha$ is sent to the
standard contact form $\alpha_0 := \ud z - \lambda$,
$z$ being the $\R$-coordinate and $\lambda$ being the
pull-back of the tautological form of $T^*\Lambda_0$.
With this local identification, for $t\in I$ close to $0$,
$\Lambda_t$ is a Legendrian graph above the zero section $\Lambda_0$,
so it is the $1$-jet of a smooth map $f_t : \Lambda_0\to\R$,
\begin{equation*}
    j^1 f_t := \{ (q,\ud f_t(q),f_t(q))\ |\ q\in\Lambda_0\}
    \subset J^1\Lambda_0.
\end{equation*}
Applying the parametrization $j(t,q):=(q,\ud f_t(q),f_t(q))$
of $(\Lambda_t)$, $t$ small enough, to Equation~(\ref{eq:LegHam}),
one then derives the Hamilton-Jacobi equation
\begin{equation}\label{eq:HamiltonJacobi}
    H_t(q,\ud f_t(q),f_t(q))
    = \frac{\partial f_t}{\partial t}(q),
\end{equation}
for all $q\in\Lambda_0$ and $t\in I$ close to $0$
(we have used that $\alpha$ is identified to $\alpha_0$ in $W$).

\begin{proof}[Proof of Lemma~\ref{lem:isotopyExtension}]
    By connectivity of the interval $I$, it is enough to prove this
    statement for $t\in I$ close enough to $0$.
    By considering a Weinstein neighborhood of $\Lambda_0$ as above,
    one can then identify $\Lambda_t$ and $\varphi_t(\Lambda_0)$ to the $1$-jet
    of the respective maps $f_t:\Lambda_0\to\R$ and $g_t:\Lambda_0\to\R$,
    for all $t\in I$, shrinking $I$ if necessary.
    Let us remark that the contact Hamiltonian map of $(\varphi_t(\Lambda_0))$
    is the restriction of $K$ to $\bigcup_t \{ t\}\times\varphi_t(\Lambda_0)$.
    Therefore, by the above discussion, both families of maps
    $(f_t)$ and $(g_t)$ satisfy the same Hamilton-Jacobi equation
    \begin{equation*}
        \frac{\partial u_t}{\partial t}(q) = K_t(q,\ud u_t(q),u_t(q)),
        \quad \forall q\in\Lambda_0, \forall t\in I,
    \end{equation*}
    with the initial condition $u_0 \equiv 0$.
    By unicity of the smooth solutions of the Hamilton-Jacobi equation
    in a small interval of time (see \emph{e.g.}
    \cite[Exercice 3.5.17]{mcduffsalamon}), $f_t = g_t$ for all $t\in I$ and
    $\Lambda_t = \varphi_t(\Lambda_0)$.
\end{proof}

\begin{cor}\label{cor:istopyExtension}
    Let $(M,\ker\alpha)$ be a contact manifold endowed
    with a complete contact form $\alpha$.
    Let $I\subset\R$ be an interval containing $0$ and
    let $(\Lambda_t)_{t\in I}$ be a (closed) Legendrian
    isotopy on $(M,\ker\alpha)$,
    the $\alpha$-contact Hamiltonian map of which is $H$.
    Then there exist contact Hamiltonians $K,G:I\times M\to\R$
    such that $\inf K = \inf H$ and $\sup G = \sup H$,
    the respective contact flows of which $(\varphi_t)$
    and $(\psi_t)$ satisfy $\varphi_t(\Lambda_0)=
    \psi_t(\Lambda_0)=\Lambda_t$ for all $t\in I$.
\end{cor}

\begin{proof}
    Let $U\subset I\times M$ be a tubular neighborhood of
    $N:=\bigcup_t \{ t\} \times\Lambda_t$ and let
    $\pi:U\to N$ be the associated smooth retraction.
    Let $\chi:I\times M\to[0,1]$ be a smooth map such that
    $\chi|_N\equiv 1$ and $\chi|_{(I\times M)\setminus U}\equiv 0$.
    The statement is now a consequence of Lemma~\ref{lem:isotopyExtension}
    applied to the following smooth extensions of $H$:
    \begin{equation*}
        K := \chi\cdot H\circ\pi + (1-\chi)\cdot\sup H \text{ and }
        G := \chi\cdot H\circ\pi + (1-\chi)\cdot\inf H.\qedhere
    \end{equation*}
\end{proof}

\subsection{The $C^1$-topology and the universal covers}
\label{se:topologyLeg}

Let $\Lambda_*$ be a closed Legendrian submanifold of $(M,\xi)$.
Let us briefly describe the $C^1$-topology of the space $\Leg(\Lambda_*)$
consisting of Legendrian submanifolds $\Lambda$ that are Legendrian isotopic to $\Lambda_*$, i.e. there exists a Legendrian isotopy $(\Lambda_t)$ such that $\Lambda_0=\Lambda_*$ and $\Lambda_1=\Lambda$,  and the topology of
the space $\Gcont(M,\xi)$ consisting contactomorphisms $\varphi$
contact isotopic to the identity through compactly supported contactomorphisms, i.e. there exists a contact isotopy $(\varphi_t)$ such that $\varphi_0=\id$, $\varphi_1=\varphi$ and $\varphi_t$ is compactly supported.
The $C^1$-topology on $\Gcont$ is understood as the topology induced
by the strong Whitney $C^1$-topology on diffeomorphisms
as it is described in \cite[Chapter~2]{Hirsch1994}.

One way to define the $C^1$-topology on $\Leg(\Lambda_*)$
is to consider $\Leg(\Lambda_*)$ as a subset of the quotient
$\mathrm{Emb}(\Lambda_*,M)/\mathrm{Diff}(\Lambda_*)$ 
of smooth embedding $\Lambda_*\hookrightarrow M$
by the right-action by diffeomorphisms of $\Lambda_*$.
This quotient being endowed with the topology induced by
the Whitney $C^1$-topology on $\mathrm{Emb}(\Lambda_*,M)$\footnote{$\Lambda_*$ being closed the strong and weak Whitney topologies agree},
the $C^1$-topology of $\Leg(\Lambda_*)$ is the induced topology
as a subset of this quotient.

Let us fix $\Lambda_0\in\Leg(\Lambda_*)$.
According to the Legendrian Weinstein neighborhood theorem,
there exists a neighborhood $W\subset M$ of $\Lambda_0$ that is
contactomorphic to a neighborhood of the $0$-section of
$J^1\Lambda_0$ endowed with the standard contact structure
and that identifies $\Lambda_0$ with the $0$-section.
Every Legendrian submanifold $\Lambda$ that is $C^1$-close to
$\Lambda_0$ is the $1$-jet of a map $f:\Lambda_0\to\R$ which is $C^2$-close to the $0$-map.
Therefore, a base of neighborhoods of the point $\Lambda_0$ in its isotopy class
$\Leg(\Lambda_*)$ is identifiable with the set
described by the balls $B_r := \{ f\in C^\infty(\Lambda_0,\R)\ |\ \| f\|_{C^2} < r
\}$, $0<r<\varepsilon$,
for a small $\varepsilon>0$.
In other words, the topological space $\Leg(\Lambda_*)$ is locally modeled
on the space of smooth maps on $\Lambda_*$ endowed with the $C^2$-topology.
Therefore, it is a locally contractible space and it admits a genuine
universal cover $\uLeg(\Lambda_*)$ that can be formally described as
a set of equivalent classes of paths $\gamma:[0,1]\to\Leg(\Lambda_*)$,
with $\gamma(0)=\Lambda_*$, with the equivalence relation given
by homotopy relative to endpoints.
With this description of $\uLeg(\Lambda_*)$, the cover
$\Pi:\uLeg\to\Leg$ is the evaluation $\gamma\mapsto\gamma(1)$.



When $M$ is closed, following a similar argument applied to the contact graphs
$\mathrm{gr}^\alpha(\psi)\subset M\times M\times\R$ of
contactomorphisms $\psi\in\Gcont(M)$,
the space $\Gcont(M)$ endowed with
the $C^1$-topology is locally modeled on the space of
compactly supported smooth maps $C^\infty(M,\R)$ endowed with
the $C^2$-topology (see \emph{e.g.} \cite[\S 1]{San13}). It is therefore a
locally contractible space
and admits a universal cover $\uGcont(M)$ that can be described
as a space of homotopy classes.
The space $\uGcont$ has a natural group structure making the cover
$\Pi:\uGcont\to\Gcont$ a morphism (coming from composing homotopies
$(\varphi_t)$ and $(\psi_t)$ timewise: $(\varphi_t\circ\psi_t)$).
The action by conjugation of contactomorphisms of $(M,\xi)$ on
$\Gcont$ naturally lift to an action on $\uGcont$:
for $g,x\in\uGcont$, $gxg^{-1}$ only depends on
$\Pi g$ and $x$. In particular, $\Pi^{-1}\{\id\}$ lies in the center of $\uGcont$.

Similarly, when $M$ is open we define $\uGcont$ (resp. $\widetilde{\cont})$ as the set of equivalence classes of $C^1$-continuous paths $\gamma : [0,1]\to \Gcont$ (resp. $\cont$) starting at the identity with the equivalence relation given by homotopy relative to endpoints.

One has a natural map $\conto(M)\times\Leg(\Lambda_*)\to\Leg(\Lambda_*)$
given by $(\varphi,\Lambda)\mapsto \varphi(\Lambda)$.
Let $I$ be an interval containing $0$,
if $(\varphi_t)_{t\in I}$ is a contact isotopy with $\varphi_0(\Lambda_*)=\Lambda_*$,
its elements act naturally on $\uLeg(\Lambda_*)$ by defining
$\varphi_t(\Lambda)$ for $\Lambda\in\uLeg(\Lambda_*)$ and $t\in I$ as the endpoint
of the path lifting $s\mapsto \varphi_{st}(\Pi\Lambda)$, $s\in[0,1]$, and
starting at $\Lambda$.
It induces in particular a continuous
map $\uGcont(M)\times\uLeg(\Lambda_*)\to\uLeg(\Lambda_*)$.

\subsubsection{Isotopies in $\uLeg$ and $\uGcont$ and their Hamiltonian maps}\label{se:uLegIsotopies}
Section~\ref{se:LegendrianIsotopies} can be naturally extended to
Legendrian isotopies of $\uLeg(\Lambda_*)$ and contact isotopies in $\uGcont$. More precisely, a family $(x_t)$ of elements of $\uLeg$  (resp. $\uGcont$) will be called
a Legendrian isotopy  (resp. a contact isotopy) if $(\Pi x_t)$ is a Legendrian isotopy (resp. a contact isotopy).  

The contact Hamiltonian map of a Legendrian isotopy of $\uLeg(\Lambda_*)$ will
designate the contact Hamiltonian map of the projected Legendrian isotopy
on $\Leg(\Lambda_*)$ and similarly for isotopies of $\uGcont$ (see Section \ref{se:LegendrianIsotopies}).

Using the natural action of contact flows on $\uLeg$
discussed above we get immediatly the following extension of Lemma~\ref{lem:isotopyExtension}.

\begin{lem}\label{lem:UisotopyExtension}
    The statement of Lemma~\ref{lem:isotopyExtension} is still true
   if the Legendrian isotopy $(\Lambda_t)$ belongs to $\uLeg(\Lambda_*)$
   for some closed Legendrian submanifold $\Lambda_*\subset M$ instead of $\Leg(\Lambda_*)$.
   In particular, the action of $\uGcont(M)$ on $\uLeg(\Lambda_*)$ is
    transitive.
\end{lem}

\subsubsection{Separability of $\Gcont,\uGcont,\Leg$ and $\uLeg$}
The topology of $\Leg$ and $\uLeg$ can
be made into the topology of a length metric space.
An auxiliary contact form $\alpha$ supporting $\xi$ and
an auxiliary Riemannian metric $g$ on $M$ being fixed,
one can define the length metric
\begin{equation}\label{eq:C1LegMetric}
    d_{C^1}(\Lambda_0,\Lambda_1) :=
    \inf_{(\Lambda_t)} \int_0^1 \| H_t\|_{C^2(\Lambda_t)} \ud t,\quad
    \forall \Lambda_0,\Lambda_1\in\Leg \text{ (resp. $\uLeg$)},
\end{equation}
where the infimum is taken over every isotopy $(\Lambda_t)$
from $\Lambda_0$ to $\Lambda_1$, the Hamiltonian of which is
denoted $H_t$, and
$\|\cdot\|_{C^2(\Lambda)}$ is the usual $C^2$-norm on maps
of the Riemannian submanifold $\Lambda\subset (M,g)$.
The equivalence between the topology induced by $d_{C^1}$
and the previously described topology is a consequence of the
equivalence of their base of neighborhoods which comes
from the Hamilton-Jacobi equation (\ref{eq:HamiltonJacobi}).

\begin{rem}
When $M$ is closed, one defines similarly a length metric on $\Gcont$ and $\uGcont$
inducing the $C^1$-topology:
\begin{equation*}\label{eq:C1HamMetric}
   d_{C^1}(\varphi,\psi) :=
   \inf_{(H_t)} \int_0^1 \| H_t\|_{C^2(M,g)} \ud t,\quad
    \forall \varphi,\psi\in\Gcont
    \text{ (resp. $\uGcont$)},
\end{equation*}
where the infimum is taken over every Hamiltonian map $(H_t)_{t\in[0,1]}$
generating a flow from $\id$ to $\psi\circ\varphi^{-1}$.
\end{rem}

The following lemmas will be useful in order to construct time-functions.

\begin{lem}\label{lem:legseparable}
    Let $(M,\xi)$ be a contact manifold and $\Lambda_*$ be a closed Legendrian submanifold. The spaces $\Leg(\Lambda_*)$ and $\uLeg(\Lambda_*)$ endowed with the $C^1$-topology are both separable.  
\end{lem}

\begin{proof}
    Recall that the topological space $\Leg(\Lambda_*)$ is seen as a subset of the base space of the projection $\mathrm{Emb}(\Lambda_*,M)\to \mathrm{Emb}(\Lambda_*,M)/\mathrm{Diff}(\Lambda_*)$. The latter projection is open since $\mathrm{Diff}(\Lambda_*)$ is a topological group acting continuously on $\mathrm{Emb}(\Lambda_*,M)$, when both spaces are endowed with the $C^1$-topology. Therefore the separability of $\Leg(\Lambda_*)$ would follow from the separability of $\mathrm{Emb}(\Lambda_*,M)$. To prove the separability of $\mathrm{Emb}(\Lambda_*,M)$, note that by Whitney embedding theorem one can see $\mathrm{Emb}(\Lambda_*,M)$ as a subset of $C^1(\Lambda_*,\R^N)$ for $N$ large enough. Moreover $C^1(\Lambda_*,\R^N)$ endowed with the $C^1$-topology is metrizable and separable (\cite[Chapter 2.1]{Hirsch1994}); therefore, as a subspace of such a topological space, $\mathrm{Emb}(\Lambda_*,M)$ is also separable, which concludes the proof of the separability of $\Leg(\Lambda_*)$. 

By \eqref{eq:C1LegMetric} and what we have just  said $\Leg(\Lambda_*)$ is metrizable and separable endowed with the $C^1$-topology, therefore it is second-countable. Thus  the Poincaré-Volterra Lemma allows to conclude that $\uLeg(\Lambda_*)$ is also second-countable and thus separable.\end{proof}

\begin{lem}\label{lem:contseparable}
   Let $(M,\xi)$ be a  closed contact manifold then the spaces $\Gcont$ and $\uGcont$ endowed with the $C^1$-topology are separable.
\end{lem}

\begin{proof}
    Since $\Gcont$ can be seen as a subset of $C^1(M,\R^N)$ for $N$ large enough by Whitney embedding theorem, one concludes as previously that $\Gcont$ is separable and metrizable. Applying the Poincaré-Volterra Lemma we deduce the separability of $\uGcont$.
\end{proof}

\subsection{Hofer type pseudo-metrics}
\label{se:Hofermetrics}

In this section, we recall the definitions of the standard generalization
of the Hofer metric of symplectic geometry to the contact setting.

Let us first recall that a \emph{pseudo-distance} $d$ on a set $X$ is
a symmetric map $d:X\times X\to [0,+\infty)$ vanishing on the diagonal
and satisfying the triangle inequality.
The pseudo-distance $d$ is \emph{non-degenerate} if $d$ is
a genuine distance, \emph{i.e.} it only vanishes on the diagonal.

We denote $\dSCH^\alpha$ the Shelukhin-Chekanov-Hofer pseudo-distance on
$\Leg$ (resp. $\uLeg$):
\begin{equation}\label{eq:dSCH}
    \dSCH^\alpha(\Lambda_0,\Lambda_1) := \inf_{(H_t)} \int_0^1 \max |H_t|\ud t,
\end{equation}
where the infimum is taken over Hamiltonian maps $H_t:\Lambda_t\to\R$,
$t\in[0,1]$, generating a Legendrian isotopy $(\Lambda_t)_{t\in[0,1]}$
from $\Lambda_0$ to $\Lambda_1$.
The Hofer oscillation pseudo-distance is defined similarly with
$\osc(H_t) := \max H_t - \min H_t$ in place of $|H_t|$:
\begin{equation}\label{eq:dHosc}
    \dHosc^\alpha(\Lambda_0,\Lambda_1) := \inf_{(H_t)} \int_0^1 \osc (H_t)\ud t.
\end{equation}
Let us remark that the Hofer oscillation pseudo-distance is clearly
degenerate as $\dHosc^\alpha(\Lambda,\phi_t^\alpha\Lambda) = 0$
for all $\Lambda$ and all $t\in\R$.

The pseudo-distance $\dSCH^\alpha$ was first studied by Rosen-Zhang in
\cite{RosenZhang2020}
where it was more generally defined on non-Legendrian subsets
(the equivalence of their definition with ours comes from
Corollary~\ref{cor:istopyExtension}).
This pseudo-distance is known to be non-degenerate in multiple
cases including when $\Leg$ is orderable \cite[Theorem~5.2]{Hedicke2022}
(see also Corollary~\ref{cor:dspec-nondeg}).
However, this pseudo-distance can also be degenerate on unorderable
$\Leg$ \cite{Cant2023}.
In the case it is degenerate on $\Leg$, $\dSCH^\alpha$ is actually
identically zero, according to the generalization of
the Chekanov's dichotomy proven by Rosen-Zhang \cite[Theorem~1.10]{RosenZhang2020}.

Given a group $G$, a \emph{pseudo-(group-)norm} $\nu$ on $G$ is
a map $G\to [0,+\infty)$ satisfying the triangle inequality
$\nu(gh)\leq \nu(g) + \nu(h)$ for all $g,h\in G$ and
the invariance under the inversion: $\nu(g)=\nu(g^{-1})$, $g\in G$.
It is non-degenerate if it only vanishes at the neutral element,
in which case $\nu$ is called a \emph{(group-)norm}.
Any group pseudo-norm $\nu$ induces a right-invariant and a left-invariant
pseudo-distance on $G$ by defining either
$(g,h)\mapsto \nu(gh^{-1})$ or $(g,h)\mapsto \nu(h^{-1}g)$.
The non-degeneracy of the induced metrics is equivalent
to the non-degeneracy of $\nu$ while
the bi-invariance of the induced pseudo-distance (which are then equal)
is equivalent to the conjugation-invariance of $\nu$.

Given a contact form $\alpha$ supporting the contact structure,
we denote $\NSH{\cdot}^\alpha$ the associated Shelukhin-Hofer pseudo-norm on
$\Gcont$
(resp. $\uGcont$):
\begin{equation}\label{eq:NSH}
    \NSH{\varphi}^\alpha := \inf_{(H_t)} \int_0^1 \max |H_t|\ud t,\quad
    \forall \varphi\in\Gcont \text{ (resp. $\uGcont$)},
\end{equation}
where the infimum is taken over Hamiltonian maps generating
a contact flow $(\varphi_t)_{t\in[0,1]}$ joining $\id$ to $\varphi$.
The Hofer oscillation pseudo-norm is defined similarly with
$\osc(H_t) := \max H_t - \min H_t$ in place of $|H_t|$:
\begin{equation}\label{eq:NHosc}
    \NHosc{\varphi}^\alpha := \inf_{(H_t)} \int_0^1 \osc (H_t)\ud t.
\end{equation}
As before, the Hofer oscillation pseudo-norm is clearly
degenerate as it vanishes on $\{ \phi_t^\alpha\}_t$.
We denote $\dSH^\alpha$ and $\dHosc^\alpha$ the respectively induced
right-invariant pseudo-distances.
These pseudo-norms were studied in depth by Shelukhin in \cite{shelukhin}.
He proved that $\NSH{\cdot}^\alpha$ is always non-degenerate on
$\Gcont$ using an energy-capacity inequality intimately linked to
non-squeezing phenomena.

In contrast to their symplectic counterparts, neither of these pseudo-distances
is invariant by the left-action of contactomorphisms, but
the following compatibility property for $g\in\cont(M,\xi)$,
$\Lambda_0,\Lambda_1\in\Leg$ and $\varphi,\psi\in\Gcont$ holds:
\begin{equation*}
    \dSCH^\alpha(g\Lambda_0,g\Lambda_1) =
    \dSCH^{g^*\alpha}(\Lambda_0,\Lambda_1) \text{ and }
    \dSH^\alpha(g\varphi,g\psi) = \dSH^{g^*\alpha}(\varphi,\psi).
\end{equation*}
The same identities hold by considering $g,\Lambda_0,\Lambda_1,\varphi$ and
$\psi$ in the universal cover of their respective spaces.
We refer to Remark~\ref{rem:invariantStructures} below for a discussion on
the apparent lack of invariance of these contact metrics.
Topologies induced by these pseudo-distances are coarser than the $C^1$-topology.

\subsection{Contact orderability}\label{se:contactorderability}

An isotopy $(\Lambda_t)$ in $\Leg$ or $\uLeg$ is said to be
non-negative (resp. positive) if its contact Hamiltonian (see Sections \ref{se:LegendrianIsotopies} and \ref{se:uLegIsotopies})
for some, and thus for any, supporting contact form is non-negative (resp. positive).
Given $\Lambda_0,\Lambda_1\in\Leg$ or $\uLeg$, we write
$\Lambda_0\cleq\Lambda_1$ (resp. $\Lambda_0 \cll\Lambda_1$)
if there exists a non-negative (resp. positive) isotopy $(\Lambda_t)$ joining $\Lambda_0$
and $\Lambda_1$.

One defines similarly relations $\cleq$ and $\cll$ on $\Gcont$ and $\uGcont$. More precisely, a contact isotopy $(\varphi_t)\subset \Gcont$ or $\uGcont$, is said to be non-negative (resp. positive) if its contact Hamiltonian (as described in Sections \ref{se:LegendrianIsotopies} and \ref{se:uLegIsotopies}) for some, and thus for any, supporting contact form is non-negative (resp. positive); and for any $\varphi,\psi\in\Gcont$ or $\uGcont$, we write $\varphi\cleq\psi$ (resp. $\varphi\cll \psi$) if there exists a non-negative (resp. positive) contact isotopy $(\varphi_t)\subset\Gcont$ or $\uGcont$ starting at $\varphi$ and ending at $\psi$. 


The relation $\cleq$ is reflexive and transitive (\emph{i.e.} it
is a pre-order), whereas $\cll$ is only transitive. 
As already remarked in \cite{EP00} on $\uGcont$, these relations on $\Gcont$ and $\uGcont$ satisfy properties of invariance with respect to the action of
contactomorphisms: we have for all $x_1,y_1,x_2,y_2\in \Gcont$ (resp. $\uGcont$)

\begin{equation}\label{eq:bi-invariant property}
    \begin{cases}
        x_1 \cleq y_1\text{ and } x_2 \cleq y_2 \Rightarrow
        x_1x_2 \cleq y_1y_2,\\
        x_1 \cll y_1\text{ and } x_2 \cleq y_2 \Rightarrow
        x_1x_2 \cll y_1y_2,\\
        x_1 \cleq y_1\text{ and } x_2 \cll y_2 \Rightarrow
        x_1x_2 \cll y_1y_2.\\
    \end{cases}
\end{equation}
\begin{prop}\label{prop:bi-inv}
The implications stated in \eqref{eq:bi-invariant property}  still hold formally when $x_1,y_1\in\Gcont$ (resp. $\uGcont)$ and $x_2,y_2\in\Leg$ (resp. $\uLeg$).
\end{prop}

\begin{proof}
We will only prove the first of the three implications stated in \eqref{eq:bi-invariant property} since the proofs of the other implications follow the same lines. 

First remark that by applying Lemma \ref{lem:isotopyExtension} (resp. Lemma \ref{lem:UisotopyExtension}), for any $x,y\in\Leg$ (resp. $\uLeg$), we have the following equivalence
 \begin{equation}\label{eq:invleg1}
     x\cleq y \Leftrightarrow \text{ there exists } g\in\Gcont \text{ (resp. } \uGcont) \text{ such that }\id\cleq g \text{ and }gx=y. 
   \end{equation}

With that in hand, let us show that the following equivalence holds
 \begin{equation}\label{eq:invleg2}
     x\cleq y \Leftrightarrow \varphi x\cleq \varphi y \text{ for all } \varphi\in\Gcont \text{ (resp. } \uGcont).
   \end{equation}
   The implication $\Leftarrow$ being obvious, we prove only the implication $\Rightarrow$. To start, suppose $x\cleq y$. Then by \eqref{eq:invleg1} there exists $g\in\Gcont$ (resp. $\uGcont$) such that $\id\cleq g$ and $g x=y$. Moreover by \eqref{eq:bi-invariant property} $\id\cleq \varphi g\varphi^{-1}$ for any $\varphi\in\Gcont$ (resp. $\uGcont$). Therefore applying the converse direction of \eqref{eq:invleg1} to $\varphi x$ we get that $\varphi x \cleq \varphi g\varphi^{-1}(\varphi x)=\varphi y$ which proves \eqref{eq:invleg2}.

   Finally take any $\varphi,\psi\in\Gcont$ (resp. $\uGcont$) such that $\varphi\cleq \psi$ and any $x,y\in\Leg$ (resp. $\uLeg$) satisfying $x\cleq y$. So by definition $\id\cleq \varphi^{-1}\psi$ which implies by \eqref{eq:invleg1} that $x\cleq \varphi^{-1}\psi x$ and by \eqref{eq:invleg2} that $\varphi x\cleq\psi x$. Using \eqref{eq:invleg1} again, $\psi x\cleq \psi y$ and so by transitivity $\varphi x\cleq \psi y$ which concludes the proof. \end{proof}

A space $\Leg$, $\uLeg$, $\Gcont$ or $\uGcont$ is said to
be \emph{orderable} if and only if $\cleq$ is antisymmetric
(\emph{i.e.} $\cleq$ is a partial order) which is equivalent that
they do not contain a non-negative and non constant loop.
Obviously, the orderability of $\Leg$ (resp. $\Gcont$) implies the
orderability of $\uLeg$ (resp. $\uGcont$), but the orderability
of the latter is much more common (see Examples~\ref{ex:orderable} below).
The notion of orderability has been introduced by Eliashberg-Polterovich
\cite{EP00} and has then been investigated by numerous authors.
A short account on these investigations is given by 
the paragraph Examples~\ref{ex:orderable} below.

\begin{rem}[On the terminology]
    Depending on the authors, the actual meaning of orderability
    may differ.
    When a contact manifold is referred to as orderable, it seems to always
    mean that either $\Gcont$ or $\uGcont$ is orderable.
    Most of the time, it means that $\uGcont$ is orderable,
    following the initial focus on $\uGcont$
    put by Eliashberg-Polterovich \cite{San11,CCR,Nakamura2023}.
    In \cite{CCR,Liu2020}, a contact manifold $(M,\xi)$ is called
    strongly orderable if $\uLeg(\Delta\times\{ 0\})$ is
    orderable (see Example~\ref{ex:orderable}.\ref{ex:it:SH}) but in
    \cite{CasPre2016,Nakamura2023} it
    means that $\Gcont$ is orderable.
    Chernov-Nemirovski call $\Leg$ (resp. $\Gcont$) universally orderable
    when $\uLeg$ (resp. $\uGcont$) is orderable \cite{chernovnemirovski1}.
\end{rem}

The following lemmata will be useful.

\begin{lem}[{\cite[Proposition~2.1.B]{EP00},
    \cite[Proposition~4.5]{chernovnemirovski2}}]\label{lem:caracteriseNonOrderability}
    Let $O$ be either $\Leg$ or $\uLeg$.
    The space $O$ is orderable if and only if there does not
    exist any positive loop among isotopies of $O$.
    The same is true for $O$ being either $\Gcont$ or
    $\uGcont$ for a \emph{closed} contact manifold.
\end{lem}

\begin{lem}\label{lem:mixedTransitivity}
    Let $O$ be either $\Leg$, $\uLeg$, $\Gcont$ or $\uGcont$
    for a given contact manifold $(M,\xi)$ and possibly a
    closed Legendrian submanifold $\Lambda_*$.
    Then
    \begin{equation*}
        \begin{cases}
            x\cleq y \text{ and } y\cll z \Rightarrow
            x\cll z,\\
            x\cll y \text{ and } y\cleq z \Rightarrow
            x\cll z,
        \end{cases}
        \quad \forall x,y,z\in O.
    \end{equation*}
\end{lem}

\begin{proof}
    Let us prove the statement for $O=\Leg$, the proof being
    similar for the other cases.
    Let us assume $\Lambda_0,\Lambda_1,\Lambda_2\in\Leg$
    are such that $\Lambda_0\cleq\Lambda_1$ and
    $\Lambda_1\cll\Lambda_2$; we want to show that
    $\Lambda_0\cll\Lambda_2$.
    By smoothly concatenating isotopies,
    one can assume the $\Lambda_i$'s, $i\in\{0,1,2\}$,
    to be part of a non-negative isotopy $(\Lambda_t)_{t\in[0,2]}$
    that is positive for $t\in (1,2]$.
    Let us now adapt a construction of Fraser-Polterovich-Rosen
    in the proof of \cite[Proposition~2.6]{FPR} in order to
    find a positive isotopy from $\Lambda_0$ to $\Lambda_2$.
    Let $v:[0,2]\to\R$ be a smooth map such that
    $v(0)=v(2)=0$ with $v'$ positive on $[0,3/2]$.
    For $\varepsilon>0$, the isotopy
    $\uLambda^\varepsilon:=(\phi_{\varepsilon v(t)}^\alpha\Lambda_t)$ 
    from $\Lambda_0$ to $\Lambda_2$ is then
    positive for $t\in [0,3/2]$.
    Indeed, by Lemma~\ref{lem:isotopyExtension}, one can
    write $\Lambda_t = \psi_t\Lambda_0$ with $(\psi_t)$
    contact flow generated by a non-negative Hamiltonian map,
    so $(\phi^\alpha_{\varepsilon v(t)}\psi_t)$
    has a $\alpha$-contact Hamiltonian $H_t \geq \varepsilon v'(t)$
    for all $t\in[0,2]$ and so does $\uLambda^\varepsilon$.
    Since the positivity of an isotopy is a $C^1$-open condition,
    the isotopy $\uLambda^\varepsilon$ is positive on $[3/2,2]$
    for $\varepsilon>0$ small enough, which brings the conclusion.
    The proof of the case $\Lambda_0\cll\Lambda_1$ and
    $\Lambda_1\cleq\Lambda_2$ is similar.
\end{proof}

\begin{cor}
    If $(x_t)_{t\in [0,1]}$ is a non-negative isotopy in either $\Leg$, $\uLeg$,
    $\Gcont$ or $\uGcont$, the Hamiltonian $(H_t)$ of which is positive
    at some $t_0\in [0,1]$, then $x_0\cll x_1$.
\end{cor}

\begin{proof}
    By continuity of $(H_t)$, there is an interval
    $[a,b]\subset [0,1]$ containing $t_0$ in its interior
    such that $H_t$ is positive for $t\in [a,b]$.
    Therefore, $x_a \cll x_b$ and the conclusion follows
    from Lemma~\ref{lem:mixedTransitivity}.
\end{proof}

\begin{exs}[Orderable and unorderable spaces]
    \label{ex:orderable}
    \
    \begin{enumerate}[1.]
        \item \label{ex:it:RP}
            $\uGcont(\RP^{2n+1})$ is orderable
            for the standard contact structure pulled-back from
            the sphere.
            Moreover $\uLeg(\RP^n)$ is orderable, where
            $\RP^n\subset\RP^{2n+1}$ is the projectivization
            of the Lagrangian subspace $\R^{n+1}\times\{ 0\}$
            of $\R^{2(n+1)}$ \cite{Giv90,EP00}.
            It generalizes to lens spaces \cite{Mil2008,San2011,GKPS}.
        \item $\uGcont(\sphere{2n+1})$ is not orderable
            for $n\geq 1$ \cite{EKP}.
        \item \label{ex:it:JN}
            Let $J^1 N=T^* N\times\R$ be the $1$-jet space of some manifold
            $N$ endowed with the standard contact structure $\ker(\ud z -
            \lambda)$, $z$ being the $\R$-coordinate and $\lambda$ the
            pull-back of the Liouville form on $T^* N$. This contact structure
            induces a contact structure on the quotient space
            $J^1N/\Z\partial_z = T^*N\times S^1$.  Let $0_N\subset J^1N$
            denote the zero-section and $p:J^1N\to J^1N/\Z\partial_z$ the
            quotient map.  Then $\Leg(0_N)$ and $\uLeg(p(0_N))$ are orderable
            when $N$ is closed (it also generalizes to compactly supported
            Legendrian isotopies when $N$ is open) \cite{bhupal,CFP,Zapolsky}.
        \item Given any contact manifold, there always are
            many Legendrian submanifolds for which $\uLeg$ is
            unorderable. Such submanifolds can be obtain by
            ``stabilizing'' any Legendrian submanifold.
            For any loose Legendrian submanifold of dimension
            $\geq 2$,
            $\uLeg$ is unorderable \cite{CFP,Liu2020}.
        \item \label{ex:it:SN}
            Given a Riemannian manifold $(N,g)$, let $SN$ denote
            the unit tangent bundle of $N$ endowed with the
            contact form $\alpha_{(x,v)}\cdot\eta
            := g(v,\ud\pi\cdot\eta)$, where $\pi:SN\to N$ is the bundle
            map.
            For any closed $N$, $\uLeg(S_x N)$ is orderable for
            any fiber $S_x N$ of $\pi$ \cite{chernovnemirovski2}.
            In particular, $\uGcont(SN)$ is orderable.
            Moreover, $\Leg(S_x N)$ is orderable, and thus so is
            $\Gcont(SN)$, when the universal
            cover of $N$ is open and $\dim N\geq 2$
            \cite{CFP,chernovnemirovski1}.
        \item \label{ex:it:SH}
            Given a closed contact manifold $(M,\xi)$ and a supporting
            contact form $\alpha$, let us consider the manifold
            $M\times M\times\R$
            endowed with the contact structure
            $\tilde{\xi}$ induced by the contact form $\tilde{\alpha}
            :=\alpha_2-e^\theta\alpha_1$,
            where $\alpha_i$ is the pull-back of $\alpha$ under the
            projection on the $i$-th factor and $\theta$ is the
            $\R$-coordinate.
            The contact form $\tilde{\alpha}$ is complete and its
            Reeb flow is given by
            \begin{equation*}
                \phi_t^{\tilde{\alpha}}(x_1,x_2,\theta) =
                (x_1,\phi_t^\alpha(x_2),\theta),\quad
                \forall (x_1,x_2,\theta)\in M\times M\times\R,
                \forall t\in\R.
            \end{equation*}
            The contact structure $\tilde{\xi}$ is independent of the
            choice of $\alpha$ supporting $\xi$, up to isomorphism.
            Let $\Delta\subset M\times M$
            be the diagonal.
            If $(M,\xi)$ is the boundary of a Liouville domain
            the symplectic homology of which
            does not vanish for some choice of coefficients,
            then $\uLeg(\Delta\times\{ 0\})$ is orderable.
            In particular, $\uGcont(M)$ is orderable
            \cite{Wei2015,albers,CCR}.
        \item A contact manifold $(M,\xi)$ is called hypertight if it
            admits a contact form having no contractible periodic
            Reeb orbit, the latter contact form is called hypertight as well.
            A Legendrian submanifold $\Lambda\subset (M,\xi)$ is called
            hypertight if there is a hypertight contact form for which
            $\Lambda$ has no contractible Reeb chord, \emph{i.e.}
            no Reeb chord in the null homotopy class of $\pi_1(M,\Lambda)$.
            If $\Lambda$ is a closed hypertight Legendrian of a
            closed contact manifold, $\uLeg(\Lambda)$ is orderable.
            If $(M,\xi)$ is a closed hypertight contact manifold,
            $\uLeg(\Delta\times\{ 0\})$
            (hence $\uGcont(M)$) is orderable for
            the contact manifold $M\times M\times\R$ defined just above
            \cite{albers,CCR}.
    \end{enumerate}
\end{exs}

\section{Order spectral selectors}
\label{se:spectralSelectors}
In this Section we first prove Theorem \ref{thm:LegSS} in Subsection \ref{se:SSLeg} and then prove Theorem \ref{thm:contSS} in Subsection \ref{se:contSS}.
\subsection{The Legendrian case}
\label{se:SSLeg}
Let $(M,\xi)$ be a cooriented contact manifold.
Let $\uLambda := (\uLambda_t)_{t\in\R}$ be a
uniformly positive path in $\Leg$
(resp. $\uLeg$), \emph{i.e.} a positive path, the $\alpha$-Hamiltonian
of which satisfies $\inf_t \min H_t > 0$ for some complete $\alpha$ supporting the
contact structure. Given $\Lambda\in\Leg$ (resp. in $\uLeg$), let
us define the spectrum of $(\Lambda,\uLambda)$ by
\begin{equation*}
    \spec(\Lambda,\uLambda) := \{ t\in\R \ | \
    \Lambda\cap\uLambda_t \neq\emptyset \},
\end{equation*}
(resp. $\spec(\Pi\Lambda,\Pi\uLambda)$).
The associated order spectral selectors are defined by
\begin{equation*}
    \ell_-(\Lambda,\uLambda) := \sup\{ t\in\R \ |\ 
    \Lambda \cgeq \uLambda_t \} \quad\text{and}\quad
    \ell_+(\Lambda,\uLambda) := \inf\{ t\in\R \ |\ 
    \Lambda \cleq \uLambda_t \}.
\end{equation*}
According to Lemma~\ref{lem:mixedTransitivity},
since $\uLambda$ is positive, one could replace
$\cleq$ (resp. $\cgeq$) in the definition of the selectors
with $\cll$ (resp. $\cgg$).
The selectors $\ell^\alpha_\pm$ correspond to the special case where
$\uLambda = (\phi_t^\alpha \Lambda_0)$.

\begin{prop}
    If $\Leg$ (resp. $\uLeg$) is orderable, the order spectral selectors
    are real-valued.
\end{prop}

\begin{proof}
  Let us prove the statement for $\Lambda$ and elements in $\uLambda$
    in the universal cover $\uLeg$ assumed orderable.
    Since $\cleq$ is a partial order,
    it amounts to proving $\uLambda_s\cleq \Lambda\cleq \uLambda_t$
    for some real numbers $s<t$.
    Let us fix a complete $\alpha$ supporting $\xi$ such that
    $\varepsilon := \inf_t \min H_t$ is positive.

    Then according to the Legendrian isotopy extension theorem
    as expressed at Lemma~\ref{lem:UisotopyExtension}
    (see also Corollary~\ref{cor:istopyExtension}),
    for every $t\geq 0$, there exists $g_t\in\tconto(M,\xi)$
    sending $\uLambda_0$ on $\uLambda_t$
    that is the time-one map of a (non necessarily compactly supported) contact flow,
    the Hamiltonian map $h:[0,1]\times M \to\R$ of which
    satisfies $\inf h = t\varepsilon$.
    Let us pick an isotopy from $\Lambda$ to $\uLambda_0$
    and extend its Hamiltonian to a compactly supported map $[0,1]\times M\to\mathbb{R}$. The induced flow $(\varphi_t)$ is compactly supported and
    $\varphi_1\Lambda = \uLambda_0$.
    
    The composition formula of Hamiltonian maps (see for example \cite[(i) Lemma 4.1]{survey})
    implies that $g_t\varphi_1$ is the time-one map of a flow
    generated by a positive  Hamiltonian
    when $t$ is taken large enough. Since $\uLambda_t=g_t\varphi_1\Lambda$,
    $\Lambda\cleq\uLambda_t$  
    when $t$ is taken large enough. 
    One proceeds similarly to prove that
    $\uLambda_s\cleq\Lambda$ for $-s>0$ large enough and
    the proof in $\Leg$ is formally identical.
\end{proof}

From now on, we always assume that $\Leg$ (resp. $\uLeg$) is orderable.

The properties of normalization, monotonicity, triangle inequalities,
Poincaré duality stated in
Theorem~\ref{thm:LegSS} directly follow from the invariance and
transitivity of $\cleq$
and the definition of the selectors $\ell_\pm^\alpha$'s.
The compatibility property is a direct consequence of the fact that
for any complete contact form $\alpha$ supporting $\xi$,
\begin{equation}\label{eq:ReebConjugation}
    g^{-1}\phi^\alpha_t g = \phi^{g^*\alpha}_t,\quad
    \forall g\in\cont(M,\xi),\forall t\in\R.
\end{equation}
The following consequence of the invariance of $\cleq$ will
be useful to restrict ourself to the spectral selectors $\ell_\pm^\alpha$
when needed.

\begin{lem}[Sign invariance]\label{lem:sign}
    For every complete contact form $\alpha$ supporting $\xi$
    and every $\Lambda\in\Leg$ (resp. $\uLeg$),
    \begin{equation*}
        \ell_\pm(\Lambda,\uLambda) < 0 \text{ (resp. $=0$, resp. $>0$)}
        \Leftrightarrow
        \ell_\pm^\alpha(\Lambda,\uLambda_0) < 0
        \text{ (resp. $=0$, resp. $>0$)}.
    \end{equation*}
    Moreover, if $\beta=e^f\alpha$ for some smooth $f:M\to\R$,
    $\forall \Lambda_0,\Lambda_1\in\Leg$ (resp. $\uLeg$),
    \begin{equation*}
        \begin{cases}
            e^{\inf f}\ell_\pm^\alpha(\Lambda_1,\Lambda_0)
            \leq \ell_\pm^\beta(\Lambda_1,\Lambda_0) \leq
            e^{\sup f}\ell_\pm^\alpha(\Lambda_1,\Lambda_0)
            &\text{when } \ell_\pm^\beta(\Lambda_1,\Lambda_0)\geq 0,\\
            e^{\sup f}\ell_\pm^\alpha(\Lambda_1,\Lambda_0)
            \leq \ell_\pm^\beta(\Lambda_1,\Lambda_0) \leq
            e^{\inf f}\ell_\pm^\alpha(\Lambda_1,\Lambda_0)
            &\text{when }\ell_\pm^\beta(\Lambda_1,\Lambda_0)\leq 0.
        \end{cases}
    \end{equation*}
\end{lem}

\begin{proof}

    One has $\ell_+(\Lambda,\uLambda)<0$ if and only if $\Lambda\cll\uLambda_0$
    so we have the first equivalence for $\ell_+$.
    Let us assume $\ell_+(\Lambda,\uLambda)>\varepsilon>0$.
    If $\ell_+^\alpha(\Lambda,\uLambda_0)=0$, then $\Lambda\cleq\phi^\alpha_t\uLambda_0$
    for all $t>0$. As $\uLambda_{\varepsilon/2}\cgg\uLambda_0$,
    there exists $t>0$ such that $\uLambda_{\varepsilon/2}\cgg\phi_t^\alpha\uLambda_0$
    so $\Lambda\cleq\uLambda_{\varepsilon/2}$, contradicting $\ell_+(\Lambda,\uLambda)
    >\varepsilon$. Conversely, if $\ell_+(\Lambda,\uLambda)=0$
    while $\ell_+^\alpha(\Lambda,\uLambda_0)>\varepsilon>0$,
    one gets a contradiction by using $\phi^\alpha_{\varepsilon/2}\uLambda_0
    \cgg \uLambda_t$ for small $t>0$.
    This implies the two other equivalences of the statement for $\ell_+$.
    The proof for $\ell_-$ is similar.

    Let us prove the second statement regarding the comparison
    of $\ell_\pm^\alpha$ and $\ell_\pm^\beta$ for $\beta:=e^f\alpha$.
    The $\beta$-Hamiltonian map of $(\phi^\alpha_t)$
    is $\beta(R^\alpha)=e^f$ while the $\beta$-Hamiltonian map
    of $(\phi^\beta_t)$ is the constant $\equiv 1$. The comparison of Hamiltonian maps
    $e^{\inf f}1\leq e^f\leq e^{\sup f}1$ implies
    \begin{equation*}
        \begin{cases}
            \phi_{e^{\inf f}t}^\beta \cleq \phi_t^\alpha \cleq
            \phi_{e^{\sup f}t}^\beta
            &\text{for } t\geq 0,\\
            \phi_{e^{\sup f}t}^\beta \cleq \phi_t^\alpha \cleq
            \phi_{e^{\inf f}t}^\beta
            &\text{for } t\leq 0,
        \end{cases}
    \end{equation*}
    which easily brings the conclusion.
\end{proof}

\begin{lem}\label{lem:HoferInequality}
    A contact form $\alpha$ supporting $\xi$ being fixed,
    if $(\Lambda_t)_{t\in[0,1]}$ is a Legendrian isotopy of
    $\Leg$ (resp. $\uLeg$), the Hamiltonian maps of which are
    denoted $H_t:\Lambda_t\to\R$, $t\in[0,1]$, one has
    \begin{equation*}
        \int_0^1 \min H_t \ud t \leq \ell_-^\alpha(\Lambda_1,\Lambda_0)
        \leq \ell_+^\alpha(\Lambda_1,\Lambda_0) \leq \int_0^1 \max H_t \ud t.
    \end{equation*}
\end{lem}

\begin{proof}
    By the isotopy extension theorem, one can find a contact Hamiltonian
    flow $(g_t)_{t\in [0,1]}$ of $(M,\xi)$ sending $\Lambda_0$
    on $\Lambda_1$, the Hamiltonian
    map $(K_t)$ of which satisfies $\max K_t = \max H_t$ for all $t\in [0,1]$.
    Let $I(t) := \int_0^t \max H_s\ud s$, the time derivative $I'$ of which
    generates the reparametrized Reeb flow $(\phi^\alpha_{I(t)})$.
    By compatibility of $\cleq$ with contactomorphisms
    (\emph{cf.} properties (\ref{eq:bi-invariant property})),
    $g_1 \cleq \phi^\alpha_{I(1)}$ implies $\Lambda_1\cleq \phi^\alpha_{I(1)}\Lambda_0$
    and the monotonicity of $\ell_+^\alpha$ together with its normalization
    property under the Reeb flow implies the last inequality of the statement.
    The first inequality is the consequence of a similar argument or the application
    of the ``Poincaré duality'' property.
\end{proof}

\begin{cor}[Continuity]\label{cor:continuity}
    A complete contact form $\alpha$ being fixed,
    \begin{equation*}
        |\ell_\pm^\alpha(\Lambda_1,\Lambda) - \ell_\pm^\alpha(\Lambda_0,\Lambda)|
        \leq \dSCH(\Lambda_0,\Lambda_1),\quad
        \forall \Lambda,\Lambda_0,\Lambda_1\in\Leg \text{ (resp. $\uLeg$)}.
    \end{equation*}
    In particular,
    the maps $\Lambda\mapsto \ell_\pm(\Lambda,\uLambda)$ are continuous
    with respect to the $C^1$-topology.
\end{cor}

This corollary is somehow reinterpreted in Corollary~\ref{cor:dspec-nondeg}
(see also inequality (\ref{eq:dspecIneq})).

\begin{proof}
    Given any Legendrian isotopy $(\Lambda_t)$ joining $\Lambda_0$ and $\Lambda_1$,
    the Hamiltonian map of which is $(H_t)$, the triangle identity together
    with Lemma~\ref{lem:HoferInequality} imply
    \begin{equation*}
        \ell_+^\alpha(\Lambda_1,\Lambda) - \ell_+^\alpha(\Lambda_0,\Lambda)
        \leq \ell_+^\alpha(\Lambda_1,\Lambda_0) \leq \int_0^1 \max H_t \ud t.
    \end{equation*}
    Intertwining $\Lambda_0$ and $\Lambda_1$ one then gets
    \begin{equation*}
        |\ell_+^\alpha(\Lambda_1,\Lambda) - \ell_+^\alpha(\Lambda_0,\Lambda)|
        \leq
        \max\left( -\int_0^1 \min H_t\ud t, \int_0^1 \max H_t\ud t\right)
        \leq \int_0^1 \max |H_t|\ud t,
    \end{equation*}
    which brings the desired inequality by taking the infimum over all
    isotopies. The analogous inequality for $\ell_-^\alpha$ follows
    by Poincaré duality or a similar proof.

    Since $\dSCH$ is $C^1$-continuous, the  maps
    $\Lambda\mapsto \ell_\pm^\alpha(\Lambda,\Lambda')$
    are $C^1$-continuous, $\Lambda'$ being fixed.
    In order to prove that a map $f:X\to\R$ is continuous,
    it is enough to prove that $f^{-1}(-\infty,x)$ and
    $f^{-1}(x,+\infty)$ are open for all $x\in\R$.
    Let $f:=\ell_\pm(\cdot,\uLambda)$ for a fixed $\uLambda$.
    Given a fixed $x\in\R$, let $\uLambda' := (\uLambda_{t+x})_{t\in\R}$,
    then
    \begin{equation*}
        f^{-1}(-\infty,x) = \{ \Lambda\ |\ \ell_+(\Lambda,\uLambda') < 0\}
        = \{ \Lambda\ |\ \ell^\alpha_+(\Lambda,\uLambda_x) < 0\},
    \end{equation*}
    where we applied the sign invariance property to some
    supporting contact form $\alpha$ at the last step.
    By $C^1$-continuity of $\ell_+^\alpha(\cdot,\uLambda_x)$, this
    last set is $C^1$-open. The other cases are treated similarly.
\end{proof}

Let us remark that the $C^1$-continuity is also a direct consequence of
the $C^1$-openness of the relations $\cll$ and $\cgg$.

\begin{prop}[Spectrality]\label{prop:spectrality}
    For every $\Lambda\in\Leg$ (resp. $\uLeg$),
    both $\ell_\pm(\Lambda,\uLambda)$ belong to $\spec(\Lambda,\uLambda)$.
\end{prop}

\begin{proof}
    Let us prove the statement in $\uLeg$.
    Let us prove that $\Pi\Lambda\cap \Pi\uLambda_t = \emptyset$
    and $\ell_+(\Lambda,\uLambda) \leq t$ implies $\ell_+(\Lambda,\uLambda)<t$;
    it would imply the result for $\ell_+$ and the result for $\ell_-$
    will follow from a similar argument.
    Since $\Pi\Lambda\cap\Pi\uLambda_t = \emptyset$, by compactness
    $\exists \varepsilon >0$, $\forall s\in (-\varepsilon,\varepsilon)$,
    $\Pi\Lambda\cap\Pi\uLambda_{t+s} = \emptyset$.
    We now essentially apply the trick
    employed by Chernov-Nemirovski in \cite[Lemma~2.2]{chernovnemirovski1}.
    By compactly extending the Hamiltonian generated by the isotopy
    $s\mapsto \uLambda_{(t+\varepsilon/2)-s\varepsilon}$, $s\in[0,1]$,
    in $[0,1]\times (M\setminus\Pi\Lambda)$,
    one finds a compactly supported $g\in\uGcont$
    sending $\uLambda_{t+\varepsilon/2}$
    to $\uLambda_{t-\varepsilon/2}$ and fixing $\Lambda$
    (according to the isotopy extension theorem
    as expressed in Lemma~\ref{lem:UisotopyExtension}).
    By definition of $\ell_+$ and positivity of $\uLambda$,
    $\Lambda \cleq \uLambda_{t+\varepsilon/2}$.
    By invariance of $\cleq$ under the action of $\uGcont$ (see Proposition \ref{prop:bi-inv}), applying
    $g$, $\Lambda \cleq \uLambda_{t-\varepsilon/2}$
    so $\ell_+(\Lambda,\uLambda)\leq t-\varepsilon/2$.

    The proof in $\Leg$ is formally the same (removing $\Pi$'s and
    tildes).
\end{proof}

\begin{lem}\label{lem:positivity}
    Given any $\Lambda\in\Leg$ (resp. $\uLeg$),
    if $\Lambda \cgeq \uLambda_0$ and $\Lambda\neq\uLambda_0$,
    then $\ell_+(\Lambda,\uLambda) > 0$.
\end{lem}

\begin{proof}
    Let us prove the statement in $\uLeg$.
    Let $\Lambda_0 := \uLambda_0$ and let $\alpha$ be a complete contact form
    supporting $\xi$.
    By sign invariance, it is enough to prove $\ell_+^\alpha(\Lambda,\Lambda_0)>0$.
    Let us consider a non-negative path $(\Lambda_t)$ from $\Lambda_0$
    to $\Lambda_1 = \Lambda$.
    Since $(\Lambda_t)$ is non-negative $\ell_+(\Lambda,\Lambda_0)\geq\ell_+(\Lambda_t,\Lambda_0)$ for all $t\in[0,1]$. Thus replacing $\Lambda$ by $\Lambda_t$ and taking $t\in(0,1]$ to be small enough one can moreover assume that $\Pi\Lambda$ belongs to a Weinstein neighborhood of $\Pi\Lambda_0$.
    Therefore, one can assume that $(M,\xi)$ is an open neighborhood
    of the zero-section of $J^1\Pi\Lambda_0$, $\Pi\Lambda_0$
    being identified with the zero-section and $\Pi\Lambda$ being
    $C^1$-close to $\Pi\Lambda_0$.
    Let $(H_t)$ be a non-negative Hamiltonian, the flow $(\psi^t)$
    of which satisfies $\psi^t\Lambda_0 = \Lambda_t$ for all $t\in[0,1]$.
    Since $\Pi\Lambda_0\neq\Pi\Lambda$, there exists $t_0\in[0,1)$, and
    a non-empty neighborhood $U\subset\Pi\Lambda_0$, such that
    $H_{t_0}\circ\psi^{t_0}(q)>0$ for $q\in U$.
    One can
    furthermore assume without loss of generality $t_0 = 0$ 
    (by replacing $\Lambda_0$ with $\psi^{t_0}\Lambda_0$,
    $U$ with $\psi^{t_0}U$ etc.) so that
    $H_0(q)>0$ for all $q\in U$.
    
    We will adapt a procedure due to Eliashberg-Polterovich
    (see the proof of \cite[Proposition~2.1.B]{EP00}).
    Let $(\varphi_i)_{1\leq i\leq n}$ be a finite family of diffeomorphisms
    of $\Pi\Lambda_0$
    isotopic to the identity such that $(\varphi_i(U))$ covers $\Pi\Lambda_0$
    (it exists by closeness of $\Pi\Lambda_0$).
    It lifts to $J^1\Pi\Lambda_0$ as a family of contactomorphisms
    isotopic to identity and preserving the zero-section:
    associating to a diffeomorphism $\varphi$, the contactomorphism
    $(q,p,z)\mapsto (\varphi(q),p\circ\ud\varphi^{-1},z)$.
    By cutting-off their Hamiltonian maps away from the zero section,
    one gets contactomorphisms $(g_i)_{1\leq i\leq n}$ of $(M,\xi)$
    fixing $\Lambda_0$ such that $(g_i(U))$ covers $\Pi\Lambda_0$.
    Let $\psi^t_i := g_i\psi^t g_i^{-1}$ for $t\in[0,1]$ and $1\leq i\leq n$.
    
    The key point is that $\psi_n\cdots\psi_1\Lambda_0\cgg\Lambda_0$,
    where $\psi_k:=\psi_k^1$.
    Indeed, it is enough to prove that the Hamiltonian
    map $(K_t)$ of the flow $(\psi_n^t\psi_{n-1}^t\cdots\psi_1^t)_t$ is positive
    along $\Pi\Lambda_0$ at time $t=0$ (see Lemma~\ref{lem:mixedTransitivity}).
    But for all $q\in \Pi\Lambda_0$,
    \begin{equation*}
            K_0(q) = 
            \alpha\left(\frac{\ud}{\ud
            t}\big(\psi_n^t\cdots\psi_1^t(q)\big)\Big|_{t=0}\right)
                   =
                   \sum_{i=1}^n \alpha\left(\dot{\psi}_i^0(q)\right),
    \end{equation*}
    where $\dot{\psi}_i^0$ stands for the time-derivative of $\psi_i^t$
    taken at time $t=0$.
    As $g_i^*\alpha = \lambda_i\alpha$ for some positive $\lambda_i:M\to (0,+\infty)$,
    $\alpha(\dot{\psi}_i^0(q))$ has the sign of
    $\alpha(\dot{\psi}^0(g_i^{-1}(q))) = H_0(g_i^{-1}(q))$.
    So each term of the summand is non-negative and the $i$-th term is
    positive when $q\in g_i(U)$.
    As $(g_i(U))$ covers $\Pi\Lambda_0$, one concludes that $K_0$ is positive
    along $\Pi\Lambda_0$.

    Therefore $\psi_n\cdots\psi_1\Lambda_0 \cgg \Lambda_0$
    so $\ell^\alpha_+(\psi_n\cdots\psi_1\Lambda_0,\Lambda_0)>0$.
    Now, by the triangle inequality,
    \begin{equation*}
        0 < \ell^\alpha_+(\psi_n\cdots\psi_1\Lambda_0,\psi_n\cdots\psi_2\Lambda_0)
        + \ell_+^\alpha(\psi_n\cdots\psi_2\Lambda_0,\psi_n\cdots\psi_3\Lambda_0)
        + \cdots + \ell_+^\alpha(\psi_n\Lambda_0,\Lambda_0),
    \end{equation*}
    so $\ell_+^\alpha(\psi_n\cdots\psi_k\Lambda_0,\psi_n\cdots\psi_{k+1}\Lambda_0)>0$
    for some $k$. By sign invariance under the contactomorphism $\psi_n\cdots\psi_{k+1}$,
    $\ell_+^\alpha(\psi_k\Lambda_0,\Lambda_0) > 0$.
    Since $g_k\Lambda_0 = \Lambda_0$, sign invariance implies
    $\ell_+^\alpha(\psi^1\Lambda_0,\Lambda_0)>0$, but
    $\psi^1\Lambda_0 = \Lambda$.

    The proof in $\Leg$ is similar.
\end{proof}

\begin{prop}[Non-degeneracy]\label{prop:nondeg}
    Given any $\Lambda\in\Leg$ (resp. $\uLeg$),
    if $\ell_-(\Lambda,\uLambda) = \ell_+(\Lambda,\uLambda) = t$,
    then $\Lambda = \uLambda_t$ (resp. for $\uLeg$ it only
    implies $\Pi\Lambda = \Pi\uLambda_t$).
\end{prop}

\begin{proof}
    Let us prove the statement in $\uLeg$,
    the case of $\Leg$ being similar.
    By contradiction, let us assume that
    $\ell_\pm(\Lambda,\uLambda)\equiv t$ and 
    $\Pi\Lambda\neq \Pi\Lambda_t$, where $\Lambda_t:=\uLambda_t$.
    By shifting the parametrization of $\uLambda$ and fixing
    a complete contact form $\alpha$,
    Lemma~\ref{lem:sign} implies $\ell_\pm^\alpha(\Lambda,\Lambda_t)\equiv 0$.
    Let $p\in \Pi\Lambda\setminus \Pi\Lambda_t$ and let us consider a
    non-negative Hamiltonian map
    $H:M\to\R$ supported outside $\Pi\Lambda_t$ and such that $H(p) > 0$,
    the flow of which is denoted $(g_s)$.
    Then $g_s\Lambda \cgeq \Lambda$ with $g_s\Lambda \neq\Lambda$ when
    $s>0$ is sufficiently small (as $g_s(p)\notin\Pi\Lambda$)
    whereas $g_s\Lambda_t = \Lambda_t$.
    By Lemma~\ref{lem:positivity}, $\ell_+^\alpha(g_s\Lambda,\Lambda) > 0$
    whereas sign invariance under $g_s$ implies
    $\ell^\alpha_+(g_s\Lambda,\Lambda_t) = \ell^\alpha_+(\Lambda,\Lambda_t) = 0$.
    This contradicts the triangle inequality
    $\ell_+^\alpha(g_s\Lambda,\Lambda) \leq \ell_+^\alpha(g_s\Lambda,\Lambda_t)
    + \ell_+^\alpha(\Lambda_t,\Lambda)$ as $\ell_+^\alpha(\Lambda_t,\Lambda)
    = - \ell_-^\alpha(\Lambda,\Lambda_t)=0$.
\end{proof}

As a corollary of Proposition~\ref{prop:spectrality} and \ref{prop:nondeg},
one gets Theorem~\ref{thm:ArnoldChord} stated in the introduction
(Section~\ref{se:existenceChords}).

\begin{proof}[Proof of Theorem~\ref{thm:ArnoldChord}]
    According to Lemma~\ref{lem:caracteriseNonOrderability},
    there exists a positive isotopy $(\Lambda_t)$ in $\Leg$ with
    $\Lambda_1=\Lambda_0$.
    Let $\Lambda\in\Leg$. By applying a contactomorphism
    sending $\Lambda$ on $\Lambda_0$, one can assume $\Lambda_0=\Lambda$.
    Let $(\widetilde{\Lambda}_t)$ be a lift of $(\Lambda_t)$ in $\uLeg$.
    By positivity, $\widetilde{\Lambda}_1\cgg\widetilde{\Lambda}_0$,
    which implies
    $\ell_\pm^\alpha(\widetilde{\Lambda}_1,\widetilde{\Lambda}_0)>0$.
    If $\ell_+^\alpha(\widetilde{\Lambda}_1,\widetilde{\Lambda}_0)>
    \ell_-^\alpha(\widetilde{\Lambda}_1,\widetilde{\Lambda}_0)$,
    one then gets two Reeb chords of distinct lengths
    by spectrality (Proposition~\ref{prop:spectrality}).
    Otherwise, one gets infinitely many distinct chords
    of the same length by non-degeneracy of the spectral selectors
    (Proposition~\ref{prop:nondeg}).
\end{proof}

\subsection{The case of contactomorphisms}\label{se:contSS}

Let us assume that $(M,\xi)$ is a closed cooriented contact
manifold such that $\Gcont$ (reps. $\uGcont$) is orderable.
Following Equations~(\ref{eq:c}), one defines
maps $c^\alpha_\pm$ for any supporting contact form $\alpha$.
Similarly to the Legendrian case, compactness of $M$ and orderability
imply that these maps are real-valued and the basic properties
of normalization, monotonicity, triangle inequalities, duality and compatibility
are straightforward consequences of the compatibility of the partial order $\cleq$
with the composition of contactomorphisms and (\ref{eq:ReebConjugation}). They
also obey a sign invariance property
analogous to Lemma~\ref{lem:sign}.

\begin{lem}[Sign invariance]\label{lem:cont-sign}
    For all contact forms $\alpha$ and $\beta$ supporting $\xi$ 
    (it is always assumed that such forms preserve the coorientation)
    and all $\varphi\in\Gcont$ (resp. $\uGcont$),
    \begin{equation*}
        c^\alpha_\pm(\varphi) < 0 \text{ (resp. $=0$, resp. $>0$)}
        \Leftrightarrow
        c^\beta_\pm(\varphi) < 0 \text{ (resp. $=0$, resp. $>0$)}.
    \end{equation*}
    More precisely, if $f:M\to\R$ is such that $\beta=e^f\alpha$,
    $\forall \varphi\in\Gcont$ (resp. $\uGcont$),
    \begin{equation*}
        \begin{cases}
            e^{\inf f}c_\pm^\alpha(\varphi)
            \leq c_\pm^\beta(\varphi) \leq
            e^{\sup f}c_\pm^\alpha(\varphi)
            &\text{when } c_\pm^\beta(\varphi)\geq 0,\\
            e^{\sup f}c_\pm^\alpha(\varphi)
            \leq c_\pm^\beta(\varphi) \leq
            e^{\inf f}c_\pm^\alpha(\varphi)
            &\text{when } c_\pm^\beta(\varphi)\leq 0.\\
        \end{cases}
    \end{equation*}
\end{lem}

The following results are consequence of the basic properties of the maps
$c^\alpha_\pm$ and are proved in a way similar to their Legendrian counterparts.

\begin{lem}\label{lem:cont-HoferInequality}
    A contact form $\alpha$ supporting $\xi$ being fixed, if
    $(\varphi_t)$ is an isotopy of $\Gcont$ (resp. $\uGcont$)
    with $\varphi_0=\id$,
    the Hamiltonian map of which is $H_t:M\to\R$, $t\in[0,1]$, one has
    \begin{equation*}
        \int_0^1 \min H_t\ud t\leq c_-^\alpha(\varphi_1) \leq
        c_+^\alpha(\varphi_1) \leq \int_0^1 \max H_t \ud t.
    \end{equation*}
\end{lem}

\begin{cor}[Continuity]\label{cor:cont-continuity}
    A supporting contact form $\alpha$ being fixed,
    \begin{equation*}
        |c_\pm^\alpha (\varphi)-c_\pm^\alpha(\psi)| \leq
        \dSH^\alpha(\varphi,\psi),\quad\forall\varphi,\psi\in\Gcont
        \text{ (resp. $\uGcont$)}.
    \end{equation*}
    In particular, the maps $c_\pm^\alpha$ are continuous with respect
    to the $C^1$-topology.
\end{cor}

\begin{rem}[Extensions of the selectors to completions]
    \label{rem:completion}
    Corollaries~\ref{cor:continuity} and \ref{cor:cont-continuity}
    allow us to naturally extend the selectors to $C^1$,
    Hofer or spectral-completions
    of the spaces $\Leg$, $\uLeg$, $\Gcont$ and $\uGcont$.
    The Hofer-1-Lipchitzness of the selectors implies their
    uniform continuity with respect to the $C^1$-metrics
    defined in Section~\ref{se:topologyLeg}.
    But we actually do not need this remark to extend the spectral
    selectors to $C^1$-contactomorphisms or Legendrian
    $C^1$-submanifolds: the definition of $\cleq$ and $\cll$
    naturally extends to these spaces and the orderability
    of the $C^1$-completion is equivalent to the orderability
    of the smooth space.

    On the other hand, it could be interesting to study the
    Hofer-completion or the spectral-completion of orderable
    $\Leg$, $\uLeg$, $\Gcont$ and $\uGcont$
    (see Section~\ref{se:spectraldistance} and
    Remark~\ref{rem:invariantStructures} below), as was initiated by
    Humilière \cite{Humiliere2008} and recently revitalized by
    Viterbo \cite{Viterbo2022} in the symplectic setting.
    It would also be interesting to compare these completions
    to the respective $C^0$-completions, which in the case of
    $\Gcont$ correspond to its $C^0$-closure inside the
    group of homeomorphism \cite{Usher2021}.
    In the symplectic setting, such a comparison can be done in some
    special cases thanks to the $C^0$-continuity of the selectors
    \cite{BHS2021}.
\end{rem}

Contrary to the Legendrian case, we only conjecture that the maps $c_\pm^\alpha$
are indeed spectral selectors while the non-degeneracy will follow from
the theorem of Tsuboi on the simplicity of the $C^1$-contactomorphisms
isotopic to the identity \cite{Tsuboi3}.

\begin{prop}\label{prop:cont-nondeg}
    Let $\psi\in\Gcont$ (resp. $\uGcont$) be such that
    $c_-^\alpha(\psi) = c_+^\alpha(\psi) = t$ for some $t\in\R$.
    Then $\psi=\phi_t^\alpha$ (resp. $\Pi\psi = \phi_t^\alpha$).
\end{prop}

\begin{proof}
    Let us first deal with the case of $\Gcont$.
    We want to apply the theorem of Tsuboi asserting that
    the group of $C^1$-contactomorphisms isotopic to identity
    $\Gcont^1$ is a simple group.
    The group of $C^1$-contactomorphisms is defined by Tsuboi in \cite[\S 2]{Tsuboi3},
    it contains $\Gcont$ and its identity component
    is contained in the $C^1$-completion
    of $\Gcont$: $\varphi$ is a $C^1$-contactomorphism
    if it is a $C^1$-diffeomorphism
    such that $\varphi^*\alpha = e^f \alpha$
    for a $C^1$-map $f:M\to\R$.
    As $\Gcont^1$ is contained in the $C^1$-completion of $\Gcont$,
    we endow it with the $C^1$-topology and
    the maps $c_\pm^\alpha$ naturally extend to $C^1$-continuous maps
    $\bar{c}_\pm^\alpha:\Gcont^1\to\R$
    (see Remark~\ref{rem:completion} just above).
    Let us define the subset $Z^1\subset\Gcont^1$
    \begin{equation*}
        Z^1 := \{ \varphi\in\Gcont^1 \ |\ \bar{c}_-^\alpha(\varphi) =
            \bar{c}_+^\alpha(\varphi) = 0 \},
    \end{equation*}
    and denote $Z:= Z^1\cap\Gcont$.
    The subset $Z^1$ is in fact a normal subgroup of $\Gcont^1$ (resp. $\uGcont$).
    It is indeed a subgroup: $\id\in Z^1$ by the normalization property, if
    $\varphi,\psi\in Z^1$,
    then $\varphi^{-1}\in Z^1$ by the ``Poincaré duality'' property, while
    $\varphi\psi\in Z^1$ by applying both triangle inequalities and
    $\bar{c}_-^\alpha \leq \bar{c}_+^\alpha$.

    In order to prove that $Z^1$ is normal, let us prove that
    $\bar{c}^\alpha_\pm(\varphi)=0$ implies
    $\bar{c}^\alpha_\pm(g\varphi g^{-1})=0$ (for $\bar{c}^\alpha_+$ or
    $\bar{c}^\alpha_-$ independently) for all $g,\varphi\in\Gcont^1$.
    Since $c^\alpha_\pm(g\varphi g^{-1})=c^{g^*\alpha}_\pm(\varphi)$
    for $g,\varphi\in\Gcont$, by Lemma~\ref{lem:cont-sign},
    \begin{equation}\label{eq:comparisonConjugation}
        e^{-\sup|f_g|}c^\alpha_\pm(\varphi)
        \leq c^\alpha_\pm(g\varphi g^{-1}) \leq
        e^{\sup|f_g|} c^\alpha_\pm(\varphi),
    \end{equation}
    where $f_g:M\to\R$ is the smooth function such that
    $g^*\alpha = e^{f_g}\alpha$. Let us recall that
    $g^*\alpha = e^{f_g}\alpha$ for $f_g$ of class $C^1$
    when $g\in\Gcont^1$ so that
    $h\mapsto \sup |f_h|$ is a continuous map $\Gcont\to\R$ 
    that extends to $\Gcont^1$.
    Therefore the identity (\ref{eq:comparisonConjugation}) extends
    to those $g,\varphi$ in $\Gcont^1$ by replacing $c^\alpha_\pm$
    with its extension $\bar{c}^\alpha_\pm$.
    This extended identity implies that $Z^1$ is normal.
    The theorem of Tsuboi \cite{Tsuboi3} then implies that
    $Z^1 = \{\id\}$ as $\phi_t^\alpha\notin Z^1$ for $t\neq 0$ by the normalization property of Theorem \ref{thm:contSS}.
    So $Z=\{\id\}$ which brings the conclusion.

    In the case of $\uGcont$, one needs to also consider the
    universal cover $\uGcont^1$ of $\Gcont^1$
    (it is a genuine universal cover since $\Gcont^1$
    is locally contractible \cite[\S 3]{Tsuboi3}).
    We denote $\Pi^1$ the cover map (extending the cover map $\Pi$) and
    define $Z^1\subset\uGcont^1$ as the intersection of the zero
    sets of $\bar{c}_\pm^\alpha$ as above.
    As Before, $Z^1$ is a normal subgroup of $\uGcont^1$.
    As $\Pi^1$ is a surjective group morphism,
    $\Pi^1Z^1$ is a normal subgroup of $\Gcont^1$ which
    does not contain $\phi_t^\alpha$ for $t\neq 0$ small enough.
    The theorem of Tsuboi then implies $\Pi^1 Z^1 = \{ \id\}$,
    which allows us to conclude.
\end{proof}

There is a counterpart to Lemma~\ref{lem:positivity} to the case 
of contactomorphisms that can be proved using the same procedure
inspired by Eliashberg-Polterovich. However, in the current case,
it can also be seen as a consequence of Proposition~\ref{prop:cont-nondeg}.

\begin{cor}\label{cor:cont-positivity}
    Given $\varphi\in\Gcont$ (resp. $\uGcont$),
    if $\varphi\cgeq\id$ and $\varphi\neq\id$, then
    $c^\alpha_+(\varphi) > 0$ for any supporting $\alpha$.
\end{cor}

\begin{proof}
    In the case $\varphi\in\Gcont$, by monotonicity of $c^\alpha_\pm$
    and normalization, $c^\alpha_+(\varphi)\geq c^\alpha_-(\varphi)\geq 0$.
    Therefore, $c_+^\alpha(\varphi) = 0$ would imply $c_-^\alpha(\varphi)=0$
    so $\varphi=\id$ by Proposition~\ref{prop:cont-nondeg}.
    In the case $\varphi\in\uGcont$, since $\varphi\cgeq \id$, there
    exists a non-negative isotopy $(\varphi_t)$ from $\id$ to $\varphi_1 = \varphi$.
    Assuming furthermore $\varphi\neq\id$, this isotopy is non-constant and
    $\Pi\varphi_t \neq \Pi\id$ for some $t\in (0,1]$.
    By a similar argument as above, Proposition~\ref{prop:cont-nondeg} then
    implies $c_+^\alpha(\varphi_t)>0$ so $c_+^\alpha(\varphi)>0$ by monotonicity.
\end{proof}

\section{Metrics and pseudo-metrics}
\label{se:metrics}

\subsection{The spectral metrics}
\label{se:spectraldistance}

In this section, we define the natural pseudo-distance associated
with our spectral selectors, following a classical process
originated in Viterbo's seminal work \cite{Vit}, where he defined a
norm $\gamma$ on the compactly supported Hamiltonian diffeomorphisms
of $\R^{2n}$.
During the writing of this paper, the article of Nakamura \cite{Nakamura2023}
was prepublished. In this article, Nakamura defines $\dspec^\alpha$
with the aim of generating the interval topology (see
Proposition~\ref{prop:dspecTopology}).
Although the link with the $\alpha$-spectrum is out of the scope
of his work, the results of this section can also be found there.

Given a supporting contact form $\alpha$,
let us define the following pseudo-distances on $\Leg$ (resp. $\uLeg$):
$\forall \Lambda,\Lambda'\in\Leg$ (resp. $\uLeg$),
\begin{align*}
    &\text{the spectral distance} 
    &\dspec^\alpha(\Lambda,\Lambda') &:=
    \max(\ell_+^\alpha(\Lambda,\Lambda'),\ell_+^\alpha(\Lambda',\Lambda)),\\
    &\text{the gamma distance}
    &\gamma^\alpha(\Lambda,\Lambda') &:= \ell_+^\alpha(\Lambda,\Lambda')
    + \ell_+^\alpha(\Lambda',\Lambda).
\end{align*}
These pseudo-distance are symmetric, non-negative and satisfy the triangle inequality,
according to the basic properties of the spectral selectors
(the non-negativity follows from the Poincaré duality property).
The compatibility properties of the selectors also imply
\begin{equation*}
    \dspec^\alpha(\varphi(\Lambda),\varphi(\Lambda')) =
    \dspec^{\varphi^*\alpha}(\Lambda,\Lambda')
    \text{ and }
    \gamma^\alpha(\varphi(\Lambda),\varphi(\Lambda')) =
    \gamma^{\varphi^*\alpha}(\Lambda,\Lambda'),
\end{equation*}
for all $\varphi\in\cont(M,\xi)$,
making these pseudo-distances invariant under strict contactomorphisms
of $(M,\alpha)$ (see Remark~\ref{rem:invariantStructures} below for
a discussion on the apparent lack of invariance of contact metrics).
The normalization property regarding the Reeb flow also implies the
following invariance:
\begin{equation*}
    \gamma^\alpha(\phi^\alpha_t\Lambda,\Lambda') =
    \gamma^\alpha(\Lambda,\Lambda'),\quad \forall t\in\R,
\end{equation*}
so that $\gamma^\alpha$ is clearly degenerate.
Concerning non-degeneracy,
Proposition~\ref{prop:nondeg} and Lemma~\ref{lem:HoferInequality} imply
the following less direct properties.

\begin{cor}[{\cite[Theorems~2.17 and 2.26]{Nakamura2023}}]\label{cor:dspec-nondeg}
    On either $\Leg$ or $\uLeg$ orderable, one has
    \begin{equation*}
        \dspec^\alpha \leq \dSCH^\alpha \text{ and } \gamma^\alpha \leq \dHosc^\alpha,
    \end{equation*}
    for any supporting contact form $\alpha$. Moreover, $\dspec^\alpha(\Lambda,\Lambda') = 0$
    implies $\Lambda=\Lambda'$ on $\Leg$ (resp. $\Pi\Lambda=\Pi\Lambda'$ on $\uLeg$)
    while $\gamma^\alpha(\Lambda,\Lambda') = 0$ implies
    $\Lambda=\phi^\alpha_t\Lambda'$ for some $t\in\R$ on $\Leg$
    (resp. $\Pi\Lambda = \phi^\alpha_t\Pi\Lambda'$ on $\uLeg$).
\end{cor}

Nakamura's proof of the non-degeneracy properties of the Legendrian spectral metrics
follows a different path than ours: it is inspired by Rosen-Zhang work
\cite{RosenZhang2020} and provides a version of Chekanov's dichotomy.

The triangle inequality satisfied by $\ell_+^\alpha$ implies that
the maps
$\ell_\pm^\alpha(\cdot,\Lambda_0)$ are $\dspec^\alpha$-1-Lipschitz:
\begin{equation}\label{eq:dspecIneq}
    |\ell_\pm^\alpha(\Lambda,\Lambda_0)-\ell_\pm^\alpha(\Lambda',\Lambda_0)|
    \leq \dspec^\alpha(\Lambda,\Lambda'),\quad
    \forall \Lambda,\Lambda',\Lambda_0\in\Leg\text{ (resp. $\uLeg$)}.
\end{equation}
In particular, Corollary~\ref{cor:continuity} can be seen as a consequence
of Corollary~\ref{cor:dspec-nondeg}.

Similarly to the Legendrian case, given a supporting contact form $\alpha$,
one can define pseudo-distances $\dspec^\alpha$
and $\gamma^\alpha$ on $\Gcont$ (resp. $\uGcont$).
Since our space of interest is a group, it is customary to first
define the associated pseudo-norms
(see Section~\ref{se:Hofermetrics}): $\forall \varphi\in\Gcont$
(resp. $\uGcont$),
\begin{align*}
    &\text{the spectral norm}
    &\Nspec{\varphi}^\alpha &:=
    \max(c_+^\alpha(\varphi),c_+^\alpha(\varphi^{-1})),\\
    &\text{the gamma norm}
    &\gamma^\alpha(\varphi) &:= c_+^\alpha(\varphi) + c_+^\alpha(\varphi^{-1}),
\end{align*}
and then define the right-invariant pseudo-distances by $\dspec^\alpha(\varphi,\psi):=
\Nspec{\varphi\psi^{-1}}^\alpha$ and $\gamma^\alpha(\varphi,\psi) :=
\gamma^\alpha(\varphi\psi^{-1})$
(although the term ``spectral'' would suggest it actually comes from
spectral selectors, which we only conjecture for $c_\pm^\alpha$).
Although they are not left-invariant
(see however Remark~\ref{rem:invariantStructures}), they satisfy a compatibility property:
\begin{equation*}
    \dspec^\alpha(g\varphi,g\psi) = \Nspec{g\varphi\psi^{-1}g^{-1}}^\alpha
    = \Nspec{\varphi\psi^{-1}}^{g^*\alpha}
    = \dspec^{g^*\alpha}(\varphi,\psi).
\end{equation*}
These pseudo-norms and pseudo-distances share properties similar to their
Legendrian counterparts,
among which
\begin{equation}\label{eq:c-dspec_inequality}
    |c_\pm^\alpha(\varphi)-c_\pm^\alpha(\psi)| \leq \dspec^\alpha(\varphi,\psi).
\end{equation}

\begin{cor}[{\cite[Theorem~2.6]{Nakamura2023}}]\label{cor:cont-dspec-nondeg}
    On either $\Gcont$ or $\uGcont$ orderable, one has
    \begin{equation*}
        \Nspec{\cdot}^\alpha \leq \NSH{\cdot}^\alpha
        \text{ and }
        \gamma^\alpha \leq \NHosc{\cdot}^\alpha,
    \end{equation*}
    for any supporting contact form $\alpha$. Moreover,
    $\Nspec{\varphi}^\alpha = 0$ implies $\varphi=\id$ on $\Gcont$
    (resp. $\Pi\varphi = \id$ on $\uGcont$) while
    $\gamma^\alpha(\varphi) = 0$ implies $\varphi = \phi_t^\alpha$
    for some $t\in\R$ on $\Gcont$ (resp. $\Pi\varphi = \phi_t^\alpha$
    on $\uGcont$)
\end{cor}

Let us recall that a partially ordered metric space $(Z,d,\leq)$ is
a metric space $(Z,d)$ endowed with a partial order $\leq$ such that
$a\leq b\leq c$ implies $d(a,b)\leq d(a,c)$.
The following statement directly follows from the definitions.

\begin{prop}\label{prop:PartiallyOrderedMetric}
    Let $O$ be either $\Leg$, $\uLeg$, $\Gcont$ or $\uGcont$,
    associated to some contact manifold $(M,\ker\alpha)$
    (and possibly some closed Legendrian submanifold of $M$).
    If $O$ is orderable, then $(O,\dspec^\alpha,\cleq)$ is
    a partially ordered metric space.
\end{prop}

The topology induced by the spectral distance is already known
and has been introduced by Chernov-Nemirovski \cite{CheNem2020}.
It was the motivation of the recent work of Nakamura for
introducing this distance independently \cite{Nakamura2023}.

\begin{prop}[{\cite[Proposition~2.3]{Nakamura2023}}]\label{prop:dspecTopology}
    Let $O$ be either $\Leg$, $\uLeg$, $\Gcont$ or $\uGcont$,
    associated to some contact manifold $(M,\ker\alpha)$
    (and possibly some closed Legendrian submanifold of $M$)
    and let us assume that $O$ is orderable.
    The topology induced by the spectral distance $\dspec^\alpha$
    is the interval topology, that is the topology generated by
    the basis of open subsets
    \begin{equation*}
        (a,b) := \{ x\in O\ |\  a\cll x\cll b \},
    \end{equation*}
    where $a,b\in O$ satisfies $a\cleq b$.
\end{prop}

\begin{proof}
    Indeed, given $r\geq 0$ and $x\in O$, the open $\dspec^\alpha$-ball 
    $B_r(x)$ centered at $x$ of
    radius $r$ in $O$ is exactly $(\phi^\alpha_{-r}x,\phi^\alpha_r x)$ and
    these subsets form a basis of open neighborhoods of both the interval
    topology and the $\dspec^\alpha$-topology.

    Let us show precisely that $B_r(x)=(\phi^\alpha_{-r}x,\phi^\alpha_r x)$
    in the case $O=\uLeg$, the other cases being similar.
    Let $\Lambda\in\uLeg$, then $\Lambda\in B_r(x)$ is equivalent to
    $\ell_+^\alpha(\Lambda,x) < r$ and $\ell_-^\alpha(\Lambda,x) > -r$
    (applying $\ell_-^\alpha\leq\ell_+^\alpha$ and the Poincaré duality).
    Now, by definition of $\ell_+^\alpha$,
    $\ell_+^\alpha(\Lambda,x) < r$ is equivalent to
    $\Lambda \cll \phi_r^\alpha x$ while $\ell_-^\alpha(\Lambda,x)>-r$
    is equivalent to $\Lambda \cgg \phi_{-r}^\alpha x$,
    bringing the conclusion.

    The fact that the intervals $(\phi^\alpha_{-r}x,\phi^\alpha_r x)$
    form a basis of the interval topology
    is a consequence of the fact that if $x\cll y$ then
    $x\cll \phi_\varepsilon^\alpha x\cll \phi_{-\varepsilon}^\alpha y \cll y$
    for $\varepsilon>0$ small enough, which comes from the openness
    of the relation $\cll$.
\end{proof}

This proposition incidentally shows that spectral metrics induce
a same topology, any complete contact form supporting $\xi$ being given.
It can be quantified as a corollary of Lemmata~\ref{lem:sign} and
\ref{lem:cont-sign}.

\begin{cor}\label{cor:changeOfForm}
    Let $(M,\xi)$ be a contact manifold
    and $\alpha$ be a complete contact form supporting $\xi$.
    Let $f:M\to\R$ be a bounded map and $\beta := e^f\alpha$.
    Then, on either $\Leg$, $\uLeg$, $\Gcont$ or $\uGcont$
    orderable for which selectors are well-defined,
    \begin{equation*}
        e^{\inf f} g^\alpha \leq g^\beta \leq
        e^{\sup f} g^\alpha,
    \end{equation*}
    where $g^\delta$ either stands for $\Nspec{\cdot}^\delta$
    or $\dspec^\delta$ for any complete contact form $\delta$.
    In particular, on the non-vanishing set of $g^\alpha$,
    \begin{equation*}
        \big|\log |g^\alpha| - \log |g^\beta|\big| \leq d_{C^0}(\alpha,\beta),
    \end{equation*}
    with $d_{C^0}(\alpha,\beta) = \sup |f|$ where $\alpha=e^f\beta$.
\end{cor}

\begin{rem}[Invariant uniform structures and invariant metrics]
    \label{rem:invariantStructures}
    Let us assume $(M,\xi)$ is closed, for simplicity.
    Corollary~\ref{cor:changeOfForm} implies that one can not only
    endow orderable spaces $\Leg$, $\uLeg$, $\Gcont$ or $\uGcont$
    with a topology that is invariant by the natural action by contactomorphisms
    (left and right actions in the case of $\Gcont$ and $\uGcont$)
    but that this topology is induced by an invariant uniformity structure.
    Indeed, one can make sense of Cauchy sequences or
    uniform continuity by using any
    $\alpha$-spectral metric, given any auxiliary contact form $\alpha$.
    And since these spectral metrics are sent to one-another by the
    actions induced by contactomorphisms, the uniformity structure
    they induce is invariant.
    A consequence is that the spectral completion of these spaces will
    not depend on the choice of $\alpha$ either.
    In fact, a bit more is invariant: the notion of Lipschitz maps
    is also preserved as well as the notion of boundedness.
    The same story could also be applied to the Hofer-type pseudo-distance.

    One could ask whether one could have more: a natural invariant
    metric (bi-invariant in the case of $\Gcont$ and $\uGcont$). According to
    Fraser-Polterovich-Rosen \cite[Theorem~3.1]{FPR},
    such metric would essentially be discrete because of the possibility
    to squeeze a Darboux ball of $(M,\xi)$ into an arbitrarily small
    open subset. Therefore, the induced topology will not be that
    interesting, although discrete bi-invariant metrics could
    carry interesting information, especially regarding their
    asymptotic behaviors (see Sections~\ref{se:unboundedness},
    \ref{se:CSmetrics}
    and \ref{se:CSvsFPR} below).
\end{rem}

\subsection{Spectrally robust Legendrian interlinkings}\label{se:interlinked}

The notion of interlinked Legendrian submanifolds
was introduced by Entov and Polterovich
in \cite{EntPol2017}. We recall below the definition of interlinked
Legendrian submanifolds given in \cite{EntPol2022}.
Let $(M,\xi)$ be a cooriented contact manifold and $\alpha$ be a complete
contact form supporting $\xi$. An ordered pair $(\Lambda_0,\Lambda_1)$ of
disjoint Legendrian submanifolds is $\mu$-interlinked for some
positive number $\mu$ if for every
Hamiltonian map $H:\R\times M\to\R$ generating the contact flow
$(g_t)$ and satisfying $H\geq c$ 
for some $c>0$ there
exist $x\in \Lambda_0$ and $t\in (0,\mu/c]$ such that $g_t(x)\in\Lambda_1$. 
The number $\mu$ does depend on the specific choice of supporting $\alpha$.
A pair $(\Lambda_0,\Lambda_1)$ is called interlinked if it is $\mu$-interlinked
for some $\mu>0$. A pair $(\Lambda_0,\Lambda_1)$ is called $C^1$-robustly
interlinked if any pair $(\Lambda_0',\Lambda_1')$ obtained from a sufficiently
$C^1$-small Legendrian isotopy is interlinked.
We define spectral-robustness and Hofer-robustness
by replacing $C^1$-smallness with $\dspec^\alpha$ and $\dSH^\alpha$-smallness
respectively.

 \begin{thm}\label{thm:interlinked} Suppose there exists a closed Legendrian
     $\Lambda_*\subset M$ such that $\uLeg(\Lambda_*)$ is orderable. If
     $\Lambda_0,\Lambda_1\in\uLeg(\Lambda_*)$ satisfies
     $\mu:=\ell_+^\alpha(\Lambda_1,\Lambda_0)>0$, then
     $(\Pi\Lambda_0,\Pi\Lambda_1)$ is $\mu$-interlinked with respect
     to $\alpha$.
     In particular, $(\Pi\Lambda_0,\Pi\Lambda_1)$ is
     spectral-robustly interlinked
     (so Hofer and $C^1$-robustly interlinked).
 \end{thm}

In particular the hypothesis of Theorem \ref{thm:interlinked} is satisfied
whenever $\Lambda_0$ and $\Lambda_1$ are two different elements in $\uLeg$
satisfying $\Lambda_0\cleq\Lambda_1$ thanks to Lemma~\ref{lem:positivity}.
Moreover for any $\varphi\in\uGcont$ the sign invariance guarantees that
$\ell_+^\alpha(\Lambda_1,\Lambda_0)>0$ if and only if
$\ell_+^\alpha\left(\varphi(\Lambda_1),\varphi(\Lambda_0)\right)>0$
(Lemma~\ref{lem:sign}).  Therefore, Theorem~\ref{thm:interlinked} generalizes
\cite[Theorem~1.5 (i)]{EntPol2022} since $\uLeg$ of the zero-section of the $1$-jet bundle $J^1N$ is orderable for any closed manifold $N$ (see the third point in Example \ref{ex:orderable}).

Let us first prove the following lemma.

\begin{lem}\label{lem:interlinked}
    Let $\alpha$ be a complete contact form supporting the
    structure of $(M,\xi)$ and $\Lambda_*\subset M$ be a closed
    Legendrian submanifold such that $\uLeg(\Lambda_*)$ is orderable.
    Let $H:\R\times M\to\R$ be a Hamiltonian map such that
    it generates a contact flow $(g_t)_{t\in\R}$ and
    $c:=\inf H >0$.
    If $\Lambda_0,\Lambda_1\in\uLeg(\Lambda_*)$ satisfy
    $\ell_+^\alpha(\Lambda_1,\Lambda_0)>0$, then
    \begin{equation*}
        0<\ell_+\left(\Lambda_1,(g_t\Lambda_0)_{t\in\R}\right)
        \leq\frac{1}{c}\ell_+^\alpha(\Lambda_1,\Lambda_0).
    \end{equation*}
\end{lem}

\begin{proof} The first inequality comes from the sign invariance property of
    Lemma~\ref{lem:sign}.

    For the second inequality, let us remark that
    $\frac{1}{c}\ell_+^\alpha=\ell_+^{\frac{1}{c}\alpha}$ since
    $\phi_{t}^{\frac{1}{c}\alpha}=\phi_{ct}^\alpha$ for all $t\in\mathbb{R}$.
    Moreover thanks to our hypothesis on the sign of
    $\ell_+^\alpha(\Lambda_1,\Lambda_0)$ we get the following equality
    \begin{equation*}
        \frac{1}{c}\ell_+^\alpha(\Lambda_1,\Lambda_0)=
        \ell_+^{\frac{1}{c}\alpha}(\Lambda_0,\Lambda_1)=
        \inf\left\{t\geq 0
        \ |\ \phi_\alpha^{ct}\Lambda_0\cleq\Lambda_1\right\}.
    \end{equation*}
    Since $\inf H=c>0$ we deduce that
    $\phi_\alpha^{ct}\Lambda_0\cleq g_t\Lambda_0$ for all
    $t\geq 0$. Therefore, we get the desired inequality
    $\ell_+(\Lambda_1,\uLambda_0)\leq\frac{1}{c}\ell_+^\alpha(\Lambda_1,\Lambda_0)$
    by transitivity of $\cleq$.  
\end{proof}

\begin{proof}[Proof of Theorem~\ref{thm:interlinked}]
    If $\Lambda_0,\Lambda_1\in\uLeg(\Lambda_*)$ satisfies
    $\mu:=\ell_+^\alpha(\Lambda_1,\Lambda_0)>0$, for a
    Hamiltonian $H\geq c>0$ generating a contact flow $(g_t)$,
    Lemma~\ref{lem:interlinked} implies that
    $0<\ell_+\left(\Lambda_1,(g_t\Lambda_0)_{t\in\R}\right)\leq\mu/c$.
    By spectrality of $\ell_+$ (Proposition~\ref{prop:spectrality}),
    there exists $t\in (0,\mu/c]$ such that
    $\Pi\Lambda_1$ intersects $g_t\Pi\Lambda_0$, which brings the
    conclusion.

    The spectral-robustness now follows from the
    $1$-$\dspec^\alpha$-Lipchitzness of
    $\ell^\alpha_+(\Lambda_1,\cdot)$ and
    $\ell^\alpha_+(\cdot,\Lambda_0)$ 
    (by (\ref{eq:dspecIneq}) and the Poincaré duality
    property).
    It implies Hofer and $C^1$-robustness by
    Corollary~\ref{cor:dspec-nondeg}.
\end{proof}

\subsection{Unboundedness of Hofer type pseudo-metrics}
\label{se:unboundedness}

The results of this section
have also been derived by Nakamura \cite{Nakamura2023}
except for Proposition~\ref{prop:supconjshape}.

\begin{cor}
    Let $\Lambda_*\subset (M,\xi)$ be a closed Legendrian submanifold of
    a contact manifold such that $\Leg$ (resp. $\uLeg$) is orderable
    and $\alpha$ be a complete contact form supporting $\xi$.
    Then $\dSCH^\alpha(\Lambda,\Lambda')=0$ if and only if
    $\Lambda=\Lambda'$ (resp. $\Pi\Lambda=\Pi\Lambda'$),
    for any $\Lambda,\Lambda'\in\Leg$ (resp. $\uLeg$).
    Moreover, for all $\Lambda\in\Leg$ (resp. $\uLeg$),
    $(\phi_t^\alpha\Lambda)_{t\in\R}$ is a geodesic for
    $\dSCH^\alpha$:
    \begin{equation*}
        \dSCH^\alpha(\phi_t^\alpha\Lambda,
        \phi_s^\alpha\Lambda) = |t-s|,\quad
        \forall t,s\in\R,
        \forall \Lambda\in\Leg \text{ (resp. $\uLeg$)}.
    \end{equation*}
    In particular, $\dSCH^\alpha$ is unbounded.
    When $M$ is closed, the analogous statements for $\dSH^\alpha$
    on $\Gcont$ (resp. $\uGcont$)
    hold when the latter is orderable.
\end{cor}

\begin{rem}
    Note that the maps $c_\pm^{\alpha,\infty}(\psi) :=
    \lim_{n\to+\infty}\frac{c^\alpha_\pm(\psi^n)}{n}$ are
    well-defined and conjugation invariant for all
    $\psi\in\Gcont$ (resp. $\uGcont$) orderable,
    since the sequences $(c_+^\alpha(\psi^n))$ and
    $(c_-^\alpha(\psi^n))$ are respectively subadditive
    and superadditive. Moreover, they satisfy
    $c_\pm^{\alpha,\infty}(\phi^\alpha_t) = t$ for all $t\in\R$ and
    $c^\alpha_-\leq c^{\alpha,\infty}_- \leq
    c^{\alpha,\infty}_+ \leq c^\alpha_+$.
    Therefore, even the maps $\psi\mapsto \inf_{\varphi\in\Gcont}
    \Nspec{\varphi\psi\varphi^{-1}}^\alpha$ and
    $\psi\mapsto \inf_{\varphi\in\Gcont}
    \NSH{\varphi\psi\varphi^{-1}}^\alpha$ are unbounded.
\end{rem}

The non-degeneracy of $\dSCH^\alpha$ when $\Leg$ is orderable
was originally due to Hedicke \cite[Theorem~5.2]{Hedicke2022},
while the non-degeneracy of $\dSH^\alpha$ was proven by Shelukhin
without any condition on $\Gcont$ \cite{shelukhin},
as mentioned earlier in Section~\ref{se:Hofermetrics}.
The unboundedness of $\dSH^\alpha$ and $\dSCH^\alpha$ had already
been proven by Hedicke in the special case where
$(M,\xi)$ is a unit tangent bundle
with open cover and $\Lambda_*$ is a fiber of the bundle
(\emph{cf} Example~\ref{ex:orderable}.\ref{ex:it:SN})
\cite[Theorems~5.7 and 5.8]{Hedicke2022}.
For the open contact manifold $(\R^{2n}\times S^1,\ker\alpha_\mathrm{st})$
defined in Example~\ref{ex:orderable}.\ref{ex:it:JN} (here
$N:=\R^n$ and $\alpha_\mathrm{st} := \ud z-\lambda$),
some geodesics of $\dSH^{\alpha_\mathrm{st}}$ have also
been characterized in \cite{Arlove2023} using, among other things,
the spectrality of Sandon contact selectors \cite{San11}.

\begin{proof}
    According to Corollaries~\ref{cor:dspec-nondeg}
    and \ref{cor:cont-dspec-nondeg}, one has
    $\dspec^\alpha\leq\dSCH^\alpha$ on $\Leg$ or $\uLeg$
    orderable and $\dspec^\alpha\leq\dSH^\alpha$ on
    $\Gcont$ or $\uGcont$ orderable.
    The non-degeneracy statements directly follow from their
    counterpart relative to $\dspec^\alpha$.
    
    Let us prove that $(\phi_t^\alpha\Lambda)$ is a geodesic
    for $\Lambda\in\Leg$ orderable.
    By the above inequality, one gets
    \begin{equation*}
        \dSCH^\alpha(\phi_t^\alpha\Lambda,
        \phi_s^\alpha\Lambda) \geq
        \dspec^\alpha(\phi_t^\alpha\Lambda,
        \phi_s^\alpha\Lambda) = |t-s|,\quad
        \forall t,s\in\R,
    \end{equation*}
    the equality coming from the normalization property of
    $\ell_+^\alpha$.
    On the other hand, since $\phi^\alpha_{t-s}$ can be generated
    by the $\alpha$-contact Hamiltonian $H\equiv t-s$,
    $\dSCH^\alpha(\phi_t^\alpha\Lambda,\phi_s^\alpha\Lambda)
    \leq |t-s|$, for any $t,s\in\R$.
\end{proof}

For $n\geq 2$, let us consider the unit tangent bundle of the flat torus
$(\T^n:=\R^n/\Z^n,\left<\cdot,\cdot\right>)$, where
$\left<\cdot,\cdot\right>$ is the standard inner product of
$\R^n$, that can
be seen as $\T^n\times\sphere{n-1}$, that we endow with the contact form
$\alpha$ defined in
Example~\ref{ex:orderable}.\ref{ex:it:SN}. In this context, Eliashberg and
Polterovich \cite{EP00} used the Shape invariant
\cite{shapesikorav,shapeeliashberg} to construct maps $r_\pm(p,\cdot)
:\uGcont\to\R$, for all $p\in \sphere{n-1}$, which are compatible with
the order: if $\varphi\cleq \psi$ then $r_\pm(p,\varphi)\leq r_\pm(p,\psi)$.
They showed moreover that for Hamiltonian maps $H$ of the form $H(q,p):=f(p)$,
where $f:\sphere{n-1}\to\R$, one has $r_\pm(p,\phi^H_1) = f(p)$, where $\phi_1^H\in\uGcont$ is the lift of the path $(\phi_t^H)$ generated by $H$.  The
proposition was already noticed by the second author \cite[\S 5.2.2]{ArlovePhD}
and Nakamura \cite[Example~2.9]{Nakamura2023}.

\begin{prop}\label{prop:shape}
    For any $\varphi\in\uGcont$ and $p\in \sphere{n-1}$
    \[c^\alpha_-(\varphi)\leq r_\pm(p,\varphi)\leq c^\alpha_+(\varphi).\]
\end{prop}

\begin{proof}
    Let us set $c:=c_+^\alpha(\varphi)$. This implies that
    $\varphi\cleq\phi_{c+\varepsilon}^\alpha$ for any $\varepsilon>0$ and so
    $r_ \pm(p,\varphi)\leq r_\pm(p,\phi_{c+\varepsilon}^\alpha)$. Since
    $\phi_{c+\varepsilon}^\alpha$ is generated by the constant function equal
    to $c+\varepsilon$ by the result of Eliashberg and Polterovich discussed
    above $r_\pm(p,\phi_{c+\varepsilon}^\alpha)=c+\varepsilon$. Letting
    $\varepsilon$ going to $0$ we get the inequality $r_\pm(p,\varphi)\leq
    c_+^\alpha(\varphi)$. The proof for the other inequality follows the same
    lines. 
\end{proof}

\begin{cor}\label{cor:shape}
   
    Let $H:\T^n\times\sphere{n-1}\to\R$ be a Hamiltonian map of the form
    $H(q,p):=f(p)$ for some smooth $f:\sphere{n-1}\to\R$ generating
    the flow $(\phi^H_t)$.
    Then $c^\alpha_-(\phi^H_1) = \min H$ and
    $c^\alpha_+(\phi^H_1) = \max H$ so that
    \begin{equation*}
        \gamma^\alpha\left(\phi^H_1\right) = \NHosc{\phi^H_1}^\alpha =
        \osc H.
    \end{equation*}
    In particular, the gamma pseudo-norms and the Hofer oscillation
    pseudo-norms are unbounded on $\uGcont(\T^n\times\sphere{n-1},\ker\alpha)$.
\end{cor}

\begin{proof}
    As $H\leq \max H$, one has $\phi^H_1\cleq \phi^\alpha_{\max H}$ so
    $c^\alpha_+(\phi^H_1)\leq\max H$. Conversely, Proposition~\ref{prop:shape}
    implies $c^\alpha_+(\phi^H_1)\geq r_+(p,\phi^H_1) = f(p)$ for all
    $p\in\sphere{n-1}$ so that $c^\alpha_+(\phi^H_1) = \max H$.
    Similarly, $c^\alpha_-(\phi^H_1) = \min H$ so
    $\gamma^\alpha(\phi^H_1) = \osc H$.
    By definition, $\NHosc{\phi_1^H}^\alpha\leq\osc H$ so the last
    equality is a consequence of Corollary~\ref{cor:cont-dspec-nondeg}.
\end{proof}

\begin{rem}\label{rem:spectraliteconj}
    For any closed contact manifold $(M,\ker\alpha)$ a Hamiltonian map $H:
    M\to\mathbb{R}$ that is invariant under the Reeb flow (\emph{i.e.}
    $H\circ\phi_t^\alpha=H$ for all $t\in\mathbb{R}$) generates a path
    $(\phi^H_t)$ of strict contactomorphisms: $(\phi_t^H)^*\alpha=\alpha$ for
    all $t\in\R$. In this case, for any critical point $p\in M$ of $H$ one can
    easily check that $\phi_t^H(p)=\phi_{H(p)t}^\alpha(p)$. In particular
    $\{tH(p)\ |\ \ud H(p)=0\}$ is contained in $\spec^\alpha(\phi_t^H)$. Since
    in $(\T^n\times\sphere{n-1},\ker\alpha)$ the Reeb flow is the geodesic
    flow, \emph{i.e.} $\phi_t^\alpha(q,p)=(q+tp,p)$, any Hamiltonian map
    $H:\T^n\times \sphere{n-1}\to\R$ of the form $H(q,p)=f(p)$ satisfies
    $H\circ\phi_t^\alpha\equiv H$. Therefore for such Hamiltonian maps, the
    previous corollary guarantees that
    $c_\pm^\alpha(\phi_1^H)\in\spec^\alpha(\phi_1^H)$. 
\end{rem}

Since $\Nspec{\cdot}^\alpha $ is not a conjugation invariant norm, a natural
question is to ask whether 
\begin{equation*}
    \varphi\mapsto
    \sup_{\psi\in\Gcont}
\Nspec{\psi\varphi\psi^{-1}}^\alpha
\end{equation*} is a well defined conjugation invariant
norm, \emph{i.e.} whether it takes finite values or not. This question was already
asked by Shelukhin \cite[Question 18]{shelukhin} for the Hofer-Shelukhin norm
$\NSH{\cdot}^\alpha$. We do not know the answer to this question, however if we
extend the action by conjugation to the whole group of coorientation
preserving contactomorphisms that are not necessarily isotopic to the identity
we have the following proposition.  

\begin{prop}\label{prop:supconjshape} 
    Seeing the time-one map of the Reeb flow $\phi^\alpha_1$
    as an element of $\uGcont(\T^n\times\sphere{n-1},\ker\alpha)$, 
    \begin{equation*}
        \underset{\psi\in\cont}\sup
    \Nspec{\psi\phi_1^\alpha\psi^{-1}}^\alpha=\underset{\psi\in\cont}\sup\NSH{\psi\phi_1^\alpha\psi^{-1}}^\alpha=+\infty.
\end{equation*}
\end{prop}

\begin{proof}
    We recall that any diffeomorphism $\psi$ of the torus $\T^n$ can be lifted
    to a contactomorphism of $\T^n\times \sphere{n-1}$
    \begin{equation*}
    \Psi(q,p):=\left(\psi(q),\frac{\ud\psi(q)^{-T}\cdot
    p}{\left\|\ud\psi(q)^{-T}\cdot p\right\|}\right)
\end{equation*}
    whose conformal factor is $g(q,p):=-\ln\left(\left\|\ud\psi(q)^{-T}\cdot
    p\right\|\right)$, where $\ud\psi(q)^{-T}$ denotes the adjoint of
    $\ud\psi(q)^{-1}$. In particular, if $\psi\in\mathrm{GL}_n(\Z)$ is a linear
    diffeomorphism of the torus, the conformal factor $g$ of its lift is the
    function $g(q,p):=-\ln\left(\left\|\psi^{-T}\cdot p\right\|\right)$
    depending only on
    $\sphere{n-1}$.
    
    Consider the sequence of linear diffeomorphisms $(\psi_k:=\psi^k)_{k\in\N}$ where 
    \begin{equation*}
        \psi:=\left (\begin{array}{c|c|c}
                -1 & 2 & 0 \\
                \hline
                1 & -1 &0\\
                \hline
        0 & 0 & I_{n-2} \end{array}\right)\in \mathrm{GL}_n(\Z)
    \end{equation*} 
    and denote by $(\Psi_k)_{k\in\N}$ the lifted sequence of contactomorphisms.
    We note that for $k\geq 1$ the contactomorphism $\Psi_k$ is not isotopic to
    the identity since its action on the fundamental group of $\T^n\times
    \sphere{n-1}$ is given by the non trivial action of $\psi_k$ on
    $\pi_1(\T^n)=\Z^n$. We will show that $\underset{k\to+\infty}\lim
    c_+^\alpha\left(\Psi_k^{-1}\phi_1^\alpha\Psi_k\right)=+\infty$.

Indeed, $v_0:=(\sqrt{2},1,0,\ldots,0)^T\in\R^n$ is an eigenvector of
$\psi_k^{-T}=\left (\begin{array}{c|c|c}
1 & 2 & 0 \\
\hline
1 & 1 &0\\
\hline
0 & 0 & I_{n-2} \end{array}\right)^k$
associated to the eigenvalue $(1+\sqrt{2})^k$ for all $k\in\N$. Moreover the
Hamiltonian map of the path
$\left(\Psi^{-1}_k\phi_\alpha^{t}\Psi_k\right)$ is given by
$(q,p)\mapsto e^{-g_k(q,p)}=\|\psi_k^{-T}\cdot p\|$
depending only on $p$. By the result of Eliashberg and Polterovich
discussed above we deduce that
$r_\pm\left(p_0,\Psi^{-1}_k\phi_1^\alpha\Psi_k\right)=(1+\sqrt{2})^k$ where
$p_0:=\frac{v_0}{\|v_0\|}.$ Therefore
\[\Nspec{\Psi^{-1}_k\phi_1^\alpha\Psi_k}^\alpha=c_+^\alpha(\Psi^{-1}_k\phi_1^\alpha\Psi_k)\geq
(1+\sqrt{2})^k\] where the equality comes from the relation $\id\cleq
\Psi^{-1}_k\phi_1^\alpha\Psi_k$ and the inequality from Proposition
\ref{prop:shape}. Letting $k$ go to infinity implies the result for both
norms since $\NSH{\cdot}^\alpha\geq \Nspec{\cdot}^\alpha$. 
\end{proof}

\begin{rems}\ 
    \begin{enumerate}[1.]
\item Eliashberg and Polterovich showed that the maps $r_\pm(p,\cdot)$ are
    invariant under the action by conjugation of $\Gcont$ on $\uGcont$ for all
    $p\in \sphere{n-1}$. Together with Corollary \ref{cor:shape} it implies in
    particular that the conjugation invariant map
    $\gamma^\alpha_\infty:\uGcont\to\R$ defined by
    $\gamma^\alpha_\infty(\phi):=\underset{\psi\in\Gcont}
    \inf\gamma^\alpha(\psi^{-1}\phi\psi)$
    is also unbounded. However as shown in the previous proof the maps
    $r_\pm(p,\cdot)$ are not anymore invariant under the action of $\cont$ by
    conjugation.
\item It is interesting to note that if $(M,\ker\alpha)$ is a closed
    contact manifold such that the Reeb flow of $\alpha$ is $1$-periodic and
    $\uGcont(M)$ is orderable then 
\begin{equation*}
    \underset{\psi\in\cont}\sup\Nspec{\psi\varphi\psi^{-1}}^\alpha\leq
\NFPR{\varphi}<+\infty,\quad\forall\varphi\in\uGcont,
\end{equation*}
where the definition of $\NFPR{\cdot}$ is given in Section~\ref{se:CSvsFPR}.
It is not known whether there exists or not a
    closed contact manifold $(M,\xi)$ such that $\Gcont(M,\xi)$ is orderable
    and $\uGcont(M,\xi)$ is not bounded in the sense of \cite{BIP}, \emph{i.e.}
    so that there exists an unbounded conjugation invariant norm on
    $\uGcont(M,\xi)$.
\end{enumerate}
\end{rems}

\subsection{Colin-Sandon discriminant and oscillation metrics}
\label{se:CSmetrics}

Let $\Lambda_*$ be a closed Legendrian submanifold of a
cooriented contact manifold $(M,\xi)$ (not necessarily closed) and let
$\Leg := \Leg(\Lambda_*)$ (resp. $\uLeg:=\uLeg(\Lambda_*)$).
In this specific section,
given $\Lambda_0,\Lambda_1\in\Leg$ (resp. $\uLeg$),
by a path $\gamma : \Lambda_0\leadsto\Lambda_1$ we will mean
a $C^1$-continuous map $\gamma:[0,1]\to\Leg$ (resp. $[0,1]\to\uLeg$)
such that $\gamma(i)=\Lambda_i$ for $i\in\{0,1\}$.
The concatenation $\gamma_1\cdot\gamma_2$ of two paths
$\gamma_1:\Lambda_0\leadsto\Lambda_1$ and $\gamma_2:\Lambda_1\leadsto\Lambda_2$
in $\Leg$ (resp. in $\uLeg$) is by definition the path
$\Lambda_0\leadsto\Lambda_2$,
\begin{equation*}
    \gamma_1\cdot\gamma_2:t\mapsto
    \begin{cases}
        \gamma_1(2t) & t\in[0,1/2],\\
        \gamma_2(2t-1) & t\in [1/2,1].
    \end{cases}
\end{equation*}
The reverse path $\bar{\gamma}$ of a path $\gamma:\Lambda_0\leadsto\Lambda_1$
is the path $\Lambda_1\leadsto\Lambda_0$, $t\mapsto\gamma(1-t)$.
By definition of the topology defined on $\Leg$, for any path $\gamma:[0,1]\to\Leg$,
there exists a continuous map $j_\gamma:[0,1]\times\Lambda_*\to M$
such that for every $t\in[0,1]$,
$j_\gamma(t,\cdot)$ is a diffeomorphism between
$\Lambda_*$ and $\gamma(t)$.
A path $\gamma$ in $\Leg$ will be called \emph{embedded} if
there exists a smooth map $j_\gamma$ as above that
is a smooth embedding $[0,1]\times\Lambda_*\hookrightarrow M$.
A path $\gamma$ in $\uLeg$ will be called embedded if
$\Pi\gamma:t\mapsto\Pi(\gamma(t))$ is an embedded path.

In \cite[Section~8]{discriminante}, Colin and Sandon defined the
discriminant length of a path $\gamma$ in $\Leg$ (resp. $\uLeg$)
as the integral number
\begin{equation*}\label{eq:ldisc}
    \ldisc(\gamma) := \min\left\{ n\in\N\ \left|\
        \parbox{7cm}{there exist embedded paths $\gamma_1,\ldots,\gamma_n$
    such that $\gamma_1\cdots\gamma_n$ and $\gamma$ are in the same
homotopy class with fixed endpoints}\right.\right\}\in\N\bigcup\{+\infty\},
\end{equation*}
with convention $\ldisc(\gamma)=0$ if $\gamma$ is a constant map and $\min\emptyset=+\infty$. Colin-Sandon proved that $\ldisc$ takes values in $\N$ 
(with convention $0\in\N$) and
that it induces an integral metric on $\Leg$ (resp. $\uLeg$)
called the discriminant metric and defined by
\begin{equation}\label{eq:ddisc}
    \ddisc(\Lambda_0,\Lambda_1) = \min_{\gamma:\Lambda_0\leadsto\Lambda_1}
    \ldisc(\gamma),\quad \forall \Lambda_0,\Lambda_1\in\Leg
    \text{ (resp. $\uLeg$)}.
\end{equation}
Moreover, this metric is invariant under the action by contactomorphisms
of $(M,\xi)$ since embedded paths are preserved by this action.
Let us remark that in the case of the universal cover $\uLeg$,
there is a unique homotopy class of paths $\Lambda_0\leadsto\Lambda_1$
so that one can erase the minimum from Equation~(\ref{eq:ddisc}).

In addition to the discriminant length, Colin-Sandon defined
an oscillation norm (to be distinguished from the Hofer oscillation
norms). In order to properly define it, let us say that a path
$\gamma$ of $\Leg$ (resp. $\uLeg$) is monotone if it is either
a positive or a negative isotopy.
Then, for a path $\gamma$,
one defines the integral number $\losc^+(\gamma)$ by
\begin{equation*}\label{eq:loscp}
    \losc^+(\gamma) := \min\left\{ k\in\N\ \left|\
        \parbox{9cm}{there exist embedded monotone paths $\gamma_1,\ldots,\gamma_n$,
            $k$ of which are positive,
    such that $\gamma_1\cdots\gamma_n$ and $\gamma$ are in the same
homotopy class with fixed endpoints}\right.\right\},
\end{equation*}
with convention $\losc^+(\gamma)=0$ if $\gamma$ is constant
and defines $\losc^-(\gamma):=-\losc^+(\bar{\gamma})$.
Colin-Sandon proved that these numbers are finite for all paths and
they defined the oscillation norm of a path $\gamma$ by
\begin{equation*}\label{eq:losc}
    \losc(\gamma) := \losc^+(\gamma) - \losc^-(\gamma)\in\N.
\end{equation*}
Colin-Sandon proved that the induced distance $\dCSosc$
on $\Leg$ (resp. $\uLeg$) is non-degenerate if and only if $\Leg$
(resp. $\uLeg$) is orderable.
It is also a metric invariant under the action by the contactomorphisms.
We are primarily interested in the case of the universal cover $\uLeg$
for which $\dCSosc(\Lambda_0,\Lambda_1)$ is simply defined as
$\losc(\gamma)$ for any $\gamma:\Lambda_0\leadsto\Lambda_1$.
In the case of $\uLeg$, one can thus write
$\dCSosc$ as the difference $\dCSosc^+ - \dCSosc^-$ with
\begin{equation*}
    \dCSosc^\pm(\Lambda_0,\Lambda_1) := \losc^\pm(\gamma),\quad
    \forall \gamma:\Lambda_0\leadsto\Lambda_1.
\end{equation*}

In \cite{discriminante}, the unboundedness of these two distances was
only proven in the case of $\uLeg(\RP^n)$ (see
Example~\ref{ex:orderable}.\ref{ex:it:RP}), some $1$-dimensional
Legendrian knots, and $\uLeg(p(0_N))$ defined in
Example~\ref{ex:orderable}.\ref{ex:it:JN}.

\begin{thm}\label{thm:ddisc}
    Let us assume there exists a contact form $\alpha$ supporting $\xi$
    the Reeb flow of which is $T$-periodic for some $T>0$.
    If $\uLeg$ is orderable, then
    \begin{equation*}
        \begin{cases}
        \dspec^\alpha(\Lambda_0,\Lambda_1) < T\ddisc(\Lambda_0,\Lambda_1),\\
        \ell_+^\alpha(\Lambda_1,\Lambda_0) < T\dCSosc^+(\Lambda_0,\Lambda_1),
        \end{cases}
        \quad \forall \Lambda_0,\Lambda_1\in\uLeg \text{ with } \Lambda_0
        \neq\Lambda_1.
    \end{equation*}
    In particular, the discriminant metric and the
    Colin-Sandon oscillation metric are unbounded.
\end{thm}

Let us point out that the second inequality of Theorem~\ref{thm:ddisc}
implies
\begin{equation}\label{eq:dspecdCSosc}
    \dspec^\alpha(\Lambda_0,\Lambda_1) < T\dCSosc(\Lambda_0,\Lambda_1),
        \quad \forall \Lambda_0,\Lambda_1\in\uLeg \text{ with } \Lambda_0
        \neq\Lambda_1,
\end{equation}
as $\dCSosc^+(\Lambda_1,\Lambda_0)$ equals $-\dCSosc^-(\Lambda_0,\Lambda_1)$
and is always non-negative.

\begin{proof}
    Let us first prove the inequality involving the discriminant metric.
    Let $\Lambda_0\neq\Lambda$ be elements of $\uLeg$ and let
    $n:=\ddisc(\Lambda_0,\Lambda)\in\N^*$.
    By symmetry of the role of $\Lambda_0$ and $\Lambda$,
    it is enough to prove $\ell_+^\alpha(\Lambda,\Lambda_0) < nT$.

    By definition, there exists embedded paths
    $\gamma_i:\Lambda_{i-1}\leadsto\Lambda_i$, $1\leq i\leq n$,
    with $\Lambda_n=\Lambda$.
    Let us consider the maps $f_i:t\mapsto \ell_+^\alpha(\gamma_i(t),
    \gamma_i(0))$, $1\leq i\leq n$.
    By the normalization property of $\ell_+^\alpha$,
    $f_i(0) = 0$ for all $i$.
    By continuity of $\ell_+^\alpha$, the $f_i$'s are continuous
    maps.
    By spectrality of $\ell_+^\alpha$, if $f_i(t)=T$ for some $t\in(0,1]$
    and $1\leq i\leq n$, it would imply that
    $\Pi(\gamma_i(t))$ intersects $\phi^\alpha_T\Pi(\gamma_i(0))$.
    But $\phi^\alpha_T = \id$ by assumption so $\gamma_i$
    would not be an embedded path.
    Therefore, $f_i$ does not take the value $T$.
    By continuity, it implies that the $f_i$'s take their
    values in $(-T,T)$.
    Now, by the triangle inequality,
    \begin{equation*}
        \ell_+^\alpha(\Lambda,\Lambda_0) \leq
        f_1(1) + f_2(1) + \cdots + f_n(1)
        < nT,
    \end{equation*}
    the conclusion follows.

    For the second inequality, one can apply the same strategy.
    Let $k:=\dCSosc^+(\Lambda_0,\Lambda)$. By definition,
    there exists embedded monotone paths
    $\gamma_i:\Lambda_{i-1}\leadsto\Lambda_i$, $1\leq i\leq n$ for some $n\in\N$
    such that $k$ of them are positive, $n-k$ of them are negative and
    $\Lambda_n=\Lambda$.
    By the triangle inequality,
    \begin{equation}\label{eq:thmdosc}
        \ell_+^\alpha(\Lambda,\Lambda_0) \leq
        \ell_+^\alpha(\Lambda_n,\Lambda_{n-1}) +\cdots+
        \ell_+^\alpha(\Lambda_1,\Lambda_0).
    \end{equation}
    By the previous discussion, since $\gamma_i$ is embedded,
    $\ell_+^\alpha(\Lambda_i,\Lambda_{i-1}) < T$, $1\leq i\leq n$.
    Moreover, if $\gamma_i$ is negative, the monotonicity property implies
    that $\ell_+^\alpha(\Lambda_i,\Lambda_{i-1}) < 0$.
    As exactly $n-k$ of the $\gamma_i$'s
    are negative, the inequality (\ref{eq:thmdosc}) implies
    $\ell_+^\alpha(\Lambda,\Lambda_0) < kT$.

    The unboundedness of both metrics then follows from the unboundedness of
    the spectral metric (see also Equation~(\ref{eq:dspecdCSosc})): we recall
    that $\dspec^\alpha(\phi_t^\alpha\Lambda,\Lambda) = |t|$ for all $t\in\R$
    and $\Lambda\in\uLeg$.
\end{proof}

\begin{cor}\label{cor:ddiscGeodesic}
    Let us assume that there exists a contact form $\alpha$
    supporting $\xi$ the Reeb flow of which is $T$-periodic
    for some $T>0$.
    If $\Lambda\in\uLeg$ orderable is such that
    $\phi_t^\alpha\Pi\Lambda\cap\Pi\Lambda=\emptyset$ for all
    $t\notin T\Z$, then the isotopy
    $(\phi^\alpha_{tT}\Lambda)_{t\in\R}$ defines a geodesic of
    both the discriminant and the Colin-Sandon oscillation metrics, in the sense that
    \begin{equation*}
        d(\phi^\alpha_{tT}\Lambda,\phi^\alpha_{sT}\Lambda)
        = \big\lceil |t-s|\big\rceil,\quad
        \forall t,s\in\R \text{ with } t-s\notin\Z^*,
    \end{equation*}
    for $d=\ddisc$ or $\dCSosc$
    (when $t-s\in\Z^*$, the distance is $\lceil |t-s|\rceil + 1$).
\end{cor}

For $\uGcont(\R^{2n}\times S^1,\xi_\mathrm{st})$, there is a characterization
of some geodesics of the discriminant and oscillation norms of Colin-Sandon
in \cite{Arlove2023}.

\begin{proof}
    Concerning the case $t-s\notin\Z^*$,
    by invariance of both metrics under contactomorphisms, it is enough to prove
    that $d(\phi_{tT}^\alpha\Lambda,\Lambda) = \lceil t\rceil$
    for every $t\in \R_+\setminus\N$. Let us fix such a $t$
    and consider the case $d=\ddisc$.
    Then Theorem~\ref{thm:ddisc} implies that this distance is at least
    $\lceil t\rceil$.
    The assumption on $\Lambda$ implies that paths
    $s\mapsto \phi^\alpha_{t_0+s(T-\varepsilon)}\Lambda$, $s\in[0,1]$,
    are embedded for every $\varepsilon\in (0,T)$ and $t_0\in\R$.
    Let $\varepsilon := T\left(1-\frac{t}{\lceil t\rceil}\right)\in(0,T)$.
    The path $s\mapsto \phi_{stT}^\alpha\Lambda$ is homotopic
    (with fixed endpoints) to
    the concatenation $\gamma_1\cdots\gamma_{\lceil t\rceil}$ of the
    paths
    \begin{equation*}
        \gamma_i: s\mapsto \phi_{(s+i-1)(T-\varepsilon)}^\alpha\Lambda,
        \quad \forall i\in\left\{ 1,\ldots,\lceil t\rceil\right\},
    \end{equation*}
    which are embedded paths, as was just remarked.
    By definition of the discriminant metric, it implies the reverse
    inequality:
    $\ddisc(\phi_{tT}^\alpha\Lambda,\Lambda) \leq \lceil t\rceil$.

    The argument to prove $\dCSosc(\phi_{tT}^\alpha\Lambda,\Lambda)=
    \lceil t\rceil$ for $t\in\R_+\setminus\N$ is the same
    once remarked that $\dCSosc(\phi_{tT}^\alpha\Lambda,\Lambda)
    = \dCSosc^+(\phi_{tT}^\alpha\Lambda,\Lambda)$
    in order to apply Theorem~\ref{thm:ddisc}
    (since the path $s\mapsto \phi_{stT}^\alpha\Lambda$ is positive).

    Concerning the case $t-s\in\Z^*$, Theorem~\ref{thm:ddisc} implies
    the optimal lower bound since both
    metrics take integral values, while a similar 
    decomposition of the isotopy gives the reverse
    inequality.
\end{proof}

\begin{exs}
    The space $\uLeg(\RP^n)$ described at
    Example~\ref{ex:orderable}.\ref{ex:it:RP} is a space
    for which Corollary~\ref{cor:ddiscGeodesic} applies
    by taking $\Lambda:=\RP^n$ and the contact form of $\RP^{2n+1}$
    induced by the Liouville form $\frac{1}{2}(y\ud x-x\ud y)$
    on $\R^{2(n+1)}$.
    It extends to lens spaces as well.

    When $N$ is a closed manifold, $\uLeg(p(0_N))$ described
    at Example~\ref{ex:orderable}.\ref{ex:it:JN} also applies
    by taking $\Lambda:=p(0_N)$ and the standard contact form
    $\alpha=\ud z-\lambda$.

    Let us give another family of examples of such a situation.
    Consider the unit tangent bundle $M:=SN$ of a Riemannian
    manifold $N$ all of whose geodesics are closed embedded curves
    (\emph{e.g} $\sphere{n}$, $\RP^n$, $\CP^n$, $\HP^n$,
    $\CaP$ for the metric induced by the round metric on the sphere).
    For the standard choice of contact form on $M$,
    the Reeb flow corresponds to the geodesic flow and
    $\uLeg(S_xN)$ is orderable given any $x\in N$
    (\emph{cf.} Example~\ref{ex:orderable}.\ref{ex:it:SN}).
    By the geometric assumption on the geodesics of $N$,
    $\Lambda:=S_x N$ satisfies the hypothesis of Corollary~\ref{cor:ddiscGeodesic}
    for any $x\in N$.
\end{exs}

\subsection{Equivalence of the Colin-Sandon oscillation
metric and the Fraser-Polterovich-Rosen metric}
\label{se:CSvsFPR}

Let us study the natural generalization of the norm
introduced by Fraser-Polterovich-Rosen \cite{FPR} in the
context of Legendrian isotopy classes.
Let $(M,\xi)$ be a contact manifold endowed with
a contact form $\alpha$ the Reeb flow of which is $1$-periodic.
If $\uGcont$ is orderable, the Fraser-Polterovich-Rosen norm
of $\varphi\in\uGcont$ is defined as
\begin{equation*}
    \NFPR{\psi} := \min\left\{ k\in\N\ |\
    \phi^\alpha_{-k} \cleq \psi\cleq \phi^\alpha_k \right\}.
\end{equation*}
If $\uLeg$ is orderable for some closed Legendrian submanifold of $M$,
one then naturally generalizes the Fraser-Polterovich-Rosen norm
as the distance defined by
\begin{equation*}
    \dFPR(\Lambda,\Lambda') := \min\left\{ k\in\N\ |\
        \phi^\alpha_{-k}\Lambda' \cleq \Lambda\cleq
    \phi^\alpha_k\Lambda' \right\},\quad
    \forall \Lambda,\Lambda'\in\uLeg.
\end{equation*}
This is indeed a non-degenerate distance by definition of
orderability.
Moreover, this distance is invariant under the action of
the universal cover of $\conto(M,\xi)$, as
$\phi_k^\alpha\in\tconto(M,\xi)$ belongs to the center of
this group for every $k\in\Z$.
This metric also endows the partially ordered space
$(\uLeg,\cleq)$ with the structure of a partially
ordered metric space (see above Proposition~\ref{prop:PartiallyOrderedMetric}).

In the original setting of $\uGcont$,
the domination of $\dFPR$ over the oscillation metric of Colin-Sandon
has been remarked by Fraser-Polterovich-Rosen
\cite[Remark~3.8]{FPR}.
The close link between $\dFPR$ and the spectral invariant $\ell^\alpha_\pm$
allows us to prove even more in the Legendrian case.

\begin{thm}\label{thm:CSvsFPR}
    Let $(M,\xi)$ be a contact manifold endowed with a contact
    form $\alpha$ the Reeb flow of which is $1$-periodic
    and let us assume that $\uLeg(\Lambda_*)$ is orderable
    for some closed Legendrian submanifold $\Lambda_*\subset M$.
    Then the associated Colin-Sandon oscillation metric
    and Fraser-Polterovich-Rosen metric satisfy
    \begin{equation*}
        \dFPR \leq \dCSosc + 1 \leq 3A\dFPR + 1,
    \end{equation*}
    where $A:=\dCSosc(\Lambda_0,\phi^\alpha_1\Lambda_0)$
    for some $\Lambda_0\in\uLeg(\Lambda_*)$.
    In particular, the two metrics are equivalent.
\end{thm}

\begin{proof}
The spectral distance satisfies
\begin{equation*}
    \dspec^\alpha(\Lambda,\Lambda') =
    \inf \left\{ t\in [0,+\infty)\ |\
        \phi_{-t}^\alpha\Lambda'\cleq \Lambda\cleq
    \phi_t^\alpha\Lambda' \right\},\quad
    \forall\Lambda,\Lambda'\in\uLeg,
\end{equation*}
so 
\begin{equation*}
    \left\lceil \dspec^\alpha(\Lambda,\Lambda')\right\rceil
    \leq \dFPR(\Lambda,\Lambda') \leq
    \left\lfloor \dspec^\alpha(\Lambda,\Lambda')\right\rfloor + 1,
    \quad\forall\Lambda,\Lambda'\in\uLeg.
\end{equation*}
In particular, when $\dspec^\alpha$ is not an integer,
$\dFPR=\lceil\dspec^\alpha\rceil$.
The inequality $\dFPR \leq \dCSosc +1$ is then a consequence of
Theorem~\ref{thm:ddisc} (see also inequality (\ref{eq:dspecdCSosc})).

The domination of $\dCSosc$ by $\dFPR$ had already been remarked by
Fraser-Polterovich-Rosen \cite[Remark~3.8]{FPR}.
It is a consequence of the fact that $\dCSosc$
induces a partially ordered metric space on $(\uLeg,\cleq)$
(see \cite[Proposition~3.4]{discriminante}).
Let us fix $A:=\dCSosc(\Lambda_0,\phi_1^\alpha\Lambda_0)$ for
some $\Lambda_0\in\uLeg$ and let us show that
\begin{equation}\label{eq:doscLeqAk}
    \dCSosc(\Lambda,\phi_k^\alpha\Lambda) \leq A|k|,\quad
    \forall k\in\Z,\forall \Lambda\in\uLeg.
\end{equation}
First, the left-hand side of the inequality does not
depend on the choice of $\Lambda$:
let $g\in\uGcont$ such that $g\Lambda = \Lambda_0$,
then
\begin{equation*}
    \dCSosc(\Lambda,\phi_k^\alpha\Lambda) =
    \dCSosc(g\Lambda,g\phi_k^\alpha g^{-1}g\Lambda) =
    \dCSosc(\Lambda_0,\phi_k^\alpha\Lambda_0),
\end{equation*}
where we have used the $\uGcont$-invariance of $\dCSosc$
and the fact that $\phi_k^\alpha$ commutes with $g^{\pm 1}$.
The inequality (\ref{eq:doscLeqAk}) now follows from
the triangle inequality associated with the invariance
of the distance
under the action of the Reeb flow.
Now, let $\Lambda,\Lambda'\in\uLeg$ and let
$k:=\dFPR(\Lambda,\Lambda')$.
By definition, $\phi_{-k}^\alpha\Lambda'\cleq\Lambda\cleq\phi_k^\alpha\Lambda'$
so $\dCSosc(\Lambda,\phi_{-k}^\alpha\Lambda') \leq \dCSosc(\phi_{-k}^\alpha\Lambda',
\phi_k^\alpha\Lambda')$ by compatibility of $\dCSosc$ with the
partial order $\cleq$.
Therefore,
\begin{equation*}
    \begin{split}
        \dCSosc(\Lambda,\Lambda')
        &\leq \dCSosc(\Lambda,\phi_{-k}^\alpha\Lambda')
    + \dCSosc(\phi_{-k}^\alpha\Lambda',\Lambda') \\
        &\leq \dCSosc(\phi_{-k}^\alpha\Lambda',\phi_k^\alpha\Lambda')
    + Ak \\
        &\leq 3Ak,
    \end{split}
\end{equation*}
which brings the conclusion as $k=\dFPR(\Lambda,\Lambda')$.
\end{proof}

\begin{rem}
    With slight adaptations,
    one can loosen the hypothesis
    in both Theorems~\ref{thm:ddisc} and \ref{thm:CSvsFPR}
    by asking for the
    existence of a 1-periodic positive contact isotopy
    $(\phi_t)$ with $\phi_0=\id$ instead of a 1-periodic
    Reeb flow.
    One should then replace the use of $\ell_+^\alpha$ with the
    use of
    \begin{equation*}
        \ell_+^\phi(\Lambda_1,\Lambda_0) :=
        \ell_+(\Lambda_1,(\phi_t\Lambda_0)_{t\in\R}),\quad
        \forall \Lambda_0,\Lambda_1\in \uLeg,
    \end{equation*}
    and make the necessary changes in the definition of $\dspec^\alpha$
    as well as $\dFPR$.
    As $(\phi_t)$ is not an autonomous flow, the triangle inequality
    is not accessible anymore but one still has
    \begin{equation*}
        \lceil\ell_+^\phi(\Lambda_2,\Lambda_0)\rceil
        \leq 
        \lceil\ell_+^\phi(\Lambda_2,\Lambda_1)\rceil +
        \lceil\ell_+^\phi(\Lambda_1,\Lambda_0)\rceil.
    \end{equation*}
    In this case, the inequalities stated in Theorem~\ref{thm:ddisc}
    are not open anymore but it still allows to show the unboundedness
    of the metrics and Theorem~\ref{thm:CSvsFPR}.
    Moreover, for any $\Lambda\in\uLeg$ such that $\phi_t\Pi\Lambda
    \cap\Pi\Lambda=\emptyset$ for all $t\in\R\setminus\Z$,
    $(\phi_t\Lambda)$ is a geodesic in the sense of Corollary~\ref{cor:ddiscGeodesic}.
\end{rem}

\section{Lorentzian geometry and time functions}\label{se:LorentzFinsler}

When $(M,\ker\alpha)$ is a closed cooriented contact manifold, the relations
$\cleq$  and $\cll$ on $\Gcont$ (resp. $\uGcont$) have the property to come from
a closed proper cone structure, \emph{i.e.} a distribution of closed sharp convex cones
with non empty interiors \cite{fathi,minguzzi2019causality,hedicke2022causal}.
Indeed, the Lie algebra $\mathfrak{g}$ of the infinite
dimensional Lie group $\Gcont$ is the space of contact vector fields. The
subset $\mathfrak{g}_{\geq 0}$ of contact vector fields
$X\in\mathfrak{g}\setminus\{0\}$ satisfying $\alpha(X)\geq 0$ is a closed
proper cone whose interior $\mathfrak{g}_{>0}$ consists of vector fields for
which the previous inequality is open. The closed proper cone structure used to
define $\cleq$ and $\cll$ is then given by right translating these cones of the
Lie algebra to the whole tangent space.  Note that left translations would give
rise to the same closed proper cone structure since $\mathfrak{g}_{\geq 0}$ is
invariant under the adjoint action of $\Gcont$ (which is the push forward).
This point of view recently led Abbondandolo-Benedetti-Polterovich
\cite{ABP2022}  and Hedicke \cite{Hedicke2022} to
introduce objects of Lorentzian geometry to study $(\Gcont,\cleq,\cll)$ such as
Lorentz-Finsler structure,  Lorentzian distances or time functions. Recall that a continuous real valued function $\tau$ on a partially ordered topological space $(\mathcal{O},\cleq)$ is a time function if $x\cleq y$ and $x\ne y$ implies that $\tau(x)<\tau(y)$.   An open question in
\cite{ABP2022} was about the existence or not of a time
function on $(\uGcont(\RP^{2n-1}),\cleq)$. In this section,
we give a positive answer to this question generalized to all orderable
$\Gcont,\uGcont,\Leg$ and $\uLeg$.
We show moreover that time functions cannot be invariant (see Theorem \ref{thm:noInvTimeFunction} for a more precise statement).

\begin{rems}\
    \begin{enumerate}[1.]
        \item In Lorentzian geometry, the existence of a
    time function is equivalent to stable causality. The notion of stable
    causality in this context is strictly stronger than the notions of
    causality and of strong causality \cite{minguzzisanchez,minguzzi2019causality}.
        \item From a different perspective,
    Chernov-Nemirovski imported Lorentzian geometric notions to the study of some
    Legendrian isotopy classes
    \cite{chernov2011legendrian,CheNem2020,chernovnemirovski2}. 
    The starting point of their study
    is that the space of null future pointing unparametrized geodesics of a
    globally hyperbolic Lorentzian manifold carries a canonical contact
    structure.  
        \item A consequence of a recent paper of Buhovsky-Stokić
    \cite{buhovsky2023flexibility} is that the Lie algebra of the group of
    Hamiltonian symplectomorphisms of any closed symplectic manifold has no
    non-trivial invariant convex cone. It would be interesting to know if the
    only non-trivial invariant convex cones of $\mathfrak{g}$ are
    $\pm\mathfrak{g}_{\geq 0}$ and $\pm\mathfrak{g}_{>0}$ for any closed
    cooriented contact manifold $(M,\xi)$.
    \end{enumerate}
\end{rems}

Let us fix a contact form $\alpha$ supporting $\xi$ and suppose that $M$ is closed.
According to Lemma~\ref{lem:contseparable}, when $M$ is closed, endowed with the $C^1$-topology $\uGcont$ is separable
(and so is $\Gcont$).
Then, let us fix $(\psi_n)_{n\in \N}$ a dense sequence on $\Gcont$
(resp. $\uGcont$) and consider
\begin{equation*}
    \tau^\alpha(\varphi) := a\sum_n \frac{c_+^\alpha(\varphi\psi_n)}{2^n
    \max(1,|c_\pm^\alpha(\psi_n)|)} + b,\quad
    \forall \varphi\in\Gcont\text{ (resp. $\uGcont$)},
\end{equation*}
where $a,b\in\R$ are normalization factors defined by the relations
(assuming the series does converge)
\begin{equation*}
    a = \left[\sum_n
    \frac{1}{2^n\max(1,|c_\pm^\alpha(\psi_n)|)}\right]^{-1}
    \text{ and } \tau^\alpha(\id) = 0.
\end{equation*}

\begin{thm}\label{thm:cont-time-function}
 $\tau^\alpha$ is a well-defined $\dspec^\alpha$-$1$-Lipschitz
    (so $\dSH^\alpha$-$1$-Lipschitz and $C^1$-continuous) map
    $\Gcont\to\R$ (resp. $\uGcont\to\R$) satisfying
    \begin{enumerate}[1.]
        \item\label{it:tau-identity} $\tau^\alpha(\id) = 0$,
        \item\label{it:Reeb-shift} $\tau^\alpha(\phi_t^\alpha\psi) = \tau^\alpha(\psi) + t$
        for all $t\in\R$ and $\psi\in\Gcont$ (resp. $\uGcont$),
    \item\label{it:time-function} $\varphi \cleq \psi$ with $\varphi\neq\psi$ implies
        $\tau^\alpha(\varphi) < \tau^\alpha(\psi)$.
    \end{enumerate}
\end{thm}

\begin{proof}
    In order to simplify the notation let $c:=c_+^\alpha$ and let $F_n=2^n\max(1,|c_\pm^\alpha(\psi_n)|)=2^n\max(1,|c(\psi_n)|,|c(\psi_n^{-1})|)$ for all $n\in\N$.
    For any $\varphi\in\Gcont$ (resp. $\uGcont$), according to the triangle inequality,
    \begin{equation*}
        c(\varphi)-c(\psi_n^{-1}) \leq c(\varphi\psi_n) \leq
        c(\varphi)+c(\psi_n),
    \end{equation*}
    where the left-hand side inequality comes from the
    decomposition $\varphi=(\varphi\psi_n)\psi_n^{-1}$. Therefore\begin{equation}\label{eq:inégalite}
    \begin{aligned}
    |c(\varphi\psi_n)|&\leq \max\{|c(\varphi)+c(\psi_n)|,|c(\varphi)-c(\psi_n^{-1})|\}\\
    &\leq\max \{|c(\varphi)|+|c(\psi_n)|, |c(\varphi)|+|c(\psi_n^{-1})|\}.
    \end{aligned}
    \end{equation}
 Dividing \eqref{eq:inégalite}  by $F_n$ we get $\left|\frac{c(\varphi\psi_n)}{F_n}\right|\leq \frac{1}{2^n}\left(|c(\varphi)|+1\right).$ Therefore the series $\sum\limits_{n\in\N} \frac{c(\varphi\psi_n)}{F_n}$ is absolutely converging and thus  $\tau^\alpha$ as  well.

    Since $c$ is $\dspec^\alpha$-1-Lipschitz (Equation (\ref{eq:c-dspec_inequality})),
    and $\dspec^\alpha$ is right-invariant, one gets
    \begin{equation*}
        |\tau^\alpha(\varphi)-\tau^\alpha(\psi)| \leq
        a \sum_n \frac{\dspec(\varphi\psi_n,\psi\psi_n)}{
        2^n\max(1,|c_\pm^\alpha(\psi_n)|)}
        \leq \dspec(\varphi,\psi),
    \end{equation*}
    so $\tau^\alpha$ is $\dspec$-1-Lipschitz which implies that it
    is Hofer-1-Lipschitz as well as $C^1$-continuous
    according to Corollary~\ref{cor:cont-dspec-nondeg}.

    Property \ref{it:tau-identity} is true by construction while
    property \ref{it:Reeb-shift} is a direct consequence of the
    normalization property of $c$.

    Finally, suppose $\varphi\cleq\psi$ with $\varphi\neq\psi$, then
    by Corollary~\ref{cor:cont-positivity},
    $c(\psi\varphi^{-1})>2\varepsilon$ for some $\varepsilon>0$
    while $c(\varphi\varphi^{-1})=0$.
    By $C^1$-density of $(\psi_n)$ and $C^1$-continuity of $c$
    (Corollary~\ref{cor:cont-continuity}),
    $c(\varphi\circ\psi_k)<\varepsilon<c(\psi\circ\psi_k)$
    for some $k$ such that $\psi_k$ is close to
    $\varphi^{-1}$.
    Since $c$ is non-decreasing,
    $c(\varphi\circ\psi_n)\leq c(\psi\circ\psi_n)$ 
    for all $n$ and property \ref{it:time-function} follows.
\end{proof}

One defines similarly time functions on orderable $\Leg$ (resp. $\uLeg$) . More precisely, take a dense sequence $(\Lambda_n)$ of $\Leg$ (resp. $\uLeg$), whose existence is guaranteed by Lemma \ref{lem:legseparable} and fix some $\Lambda_*\in\Leg$ (resp. $\uLeg$). Consider
\begin{equation*}
    \tau^\alpha_{\Lambda_*}(\Lambda) :=
    a \sum_n \frac{\ell^\alpha_+(\Lambda,\Lambda_n)}
    {2^n\max(1,|\ell_\pm^\alpha(\Lambda_*,\Lambda_n)|)} + b,\quad
    \forall \Lambda\in\Leg \text{ (resp. $\uLeg$)},
\end{equation*}
where $a=\left[\sum_n
    \frac{1}{2^n\max(1,|\ell_\pm^\alpha(\Lambda_*,\Lambda_n)|)}\right]^{-1}$ and $b\in\R$
is such that $\tau^\alpha_{\Lambda_*}(\Lambda_*) = 0$.

One proves similarly the Legendrian counterpart of Theorem~\ref{thm:cont-time-function}

\begin{thm}\label{thm:leg-time-function}
    $\tau^\alpha_{\Lambda_*}$ is a well-defined $\dspec^\alpha$-$1$-Lipschitz
    (so $\dSCH^\alpha$-$1$-Lipschitz and $C^1$-continuous) map
    $\Leg\to\R$ (resp. $\uLeg\to\R$) satisfying
    \begin{enumerate}[1.]
        \item\label{it:leg-tau-identity} $\tau^\alpha_{\Lambda_*}(\Lambda_*) = 0$,
        \item\label{it:leg-Reeb-shift}
            $\tau^\alpha_{\Lambda_*}(\phi_t^\alpha\Lambda) =
            \tau^\alpha_{\Lambda_*}(\Lambda) + t$
        for all $t\in\R$ and $\Lambda\in\Leg$ (resp. $\uLeg$),
    \item\label{it:leg-time-function} $\Lambda \cleq \Lambda'$ with
        $\Lambda\neq\Lambda'$ implies
        $\tau^\alpha_{\Lambda_*}(\Lambda) < \tau^\alpha_{\Lambda_*}(\Lambda')$.
    \end{enumerate}
\end{thm}

Since the binary relations considered on $\Leg$ and $\uLeg$ are invariant by
the left action of $\Gcont$ and $\uGcont$ respectively, and the ones considered
on $\Gcont$ and $\uGcont$ are invariant by the action by conjugation, it is
natural to ask if there exists a time function on these spaces that can be
invariant with respect to these actions. We show that it is not possible.

\begin{thm}[Non existence of invariant time function]
    \label{thm:noInvTimeFunction}
    \ \\ \vspace{-0.5cm}
    \begin{enumerate}[1.]
    \item Let $\Leg$ (resp. $\uLeg$) such that $\Leg$ (resp. $\uLeg$) is
        orderable and $\tau$ a time function on $(\Leg,\cleq)$ (resp.
        $(\uLeg,\cleq)$). Then there exist $\Lambda,\Lambda'\in \Leg$ (resp. $\uLeg$) and
        $g\in \Gcont$ (resp. $\uGcont$) such that the time difference
        between $\Lambda$ and $\Lambda'$ is not the same as the time difference between
        $g\Lambda$ and $g\Lambda'$, \emph{i.e.} $\tau(g\Lambda)-\tau(g\Lambda')\neq
        \tau(\Lambda)-\tau(\Lambda')$.
    \item Let $O$ be either $\Gcont$ or $\uGcont$ such
        that $O$ is orderable. Let $\tau$ be a time function. Then there exist
        $\varphi,g\in O$ such that $\tau(g^{-1}\varphi g)>\tau(\varphi)$. 
\end{enumerate}
\end{thm}

\begin{proof}[Proof of Theorem \ref{thm:noInvTimeFunction} part 1]

    Let us prove it for the case $\uLeg$.
For any $\Lambda_0\in \uLeg$, by applying a Reeb flow for a small timespan,
one obtains $\Lambda_1\in\uLeg$ such that $\Pi\Lambda_1\cap
\Pi\Lambda_0=\emptyset$ and
$\Lambda_0\cleq\Lambda_1$.
We now consider a point $p\in\Pi\Lambda_0$
and a neighborhood $U$ of $p$
such that $U$ does not intersect $\Pi\Lambda_1$. Let
$h:M\to [0,+\infty)$ be a Hamiltonian map compactly supported in $U$ such that
$h(p)>0$. The induced contact flow $(g_t)$ satisfies
$\Lambda_0\cleq g_t\Lambda_0$ and
$\Lambda_0\neq g_t\Lambda_0$ for $t>0$ small enough. Moreover since
$\supp(h)\cap\Pi\Lambda_1=\emptyset$,
we deduce that $g_t\Lambda_1=\Lambda_1$. This implies that
$\tau(g_t\Lambda_0)-\tau(g_t\Lambda_1)>\tau(\Lambda_0)-\tau(\Lambda_1)$
for $t>0$ small enough.
\end{proof}

Let us remark that to prove the second part of Theorem \ref{thm:noInvTimeFunction} it is enough to construct two elements $g,\varphi\in
\Gcont$ (resp. $g\in\Gcont$ and $\varphi\in\uGcont$) such that 
\begin{equation}\label{pas d'invariance}
    \varphi\cleq g^{-1}\varphi g \text{ and } \varphi\ne g^{-1}\varphi g.
\end{equation}
We will first construct such elements when the contact manifold is the standard
Euclidean contact manifold $(\R^{2n+1},\xi_\mathrm{st})$ and then transport
this construction to any cooriented contact manifold using Darboux charts. 

\begin{proof}[Proof of Theorem \ref{thm:noInvTimeFunction} part 2]
        We will prove that there exist $g,\varphi\in\Gcont$ (resp. $\uGcont$)
        such that (\ref{pas d'invariance}) is satisfied.
        Let us first remark that it is enough to prove it in
        the standard contact vector space
        $(\R^{2n+1},\xi_\mathrm{st})$ (we recall that in an
        open manifold, $\Gcont$ stands for the set of time-one maps
        of compactly supported contact flows).
        Indeed, there would exists contact isotopies $(g_t)$, $(\varphi_t)$
        and $(h_t)$ supported in some open ball $B\subset\R^{2n+1}$ such that
        $g_0=\id=\varphi_0$, $(h_t)$ is non-negative with
        $h_0=\varphi_1$ and $h_1=g_1^{-1}\varphi_1 g_1$, and
        $\varphi_1\neq g_1^{-1}\varphi_1 g_1$.
        Now, given any contact $(2n+1)$-manifold $(M,\xi)$, there exists a contact embedding of
        $B$ inside $(M,\xi)$ by the Darboux neighborhood theorem.
        Since the contact isotopies $(g_t)$, $(\varphi_t)$ and
        $(h_t)$ are compactly supported in $B$, they naturally extend by the
        identity to contact isotopies of $(M,\xi)$.
        Seen in $(M,\xi)$, $(h_t)$ is still non-negative so that
        $\varphi:=\varphi_1$ and $g:=g_1$ still satisfy (\ref{pas d'invariance})
        as needed to conclude.

        Let us now prove the existence of such $g,\varphi\in\Gcont$
        (resp. $\uGcont$) for $(\R^{2n+1},\xi_\mathrm{st})$ to finish the proof.
        The standard contact form $\alpha_\mathrm{st}$ is $\ud z-\sum_i y_i\ud x_i$
        where $(x,y,z)$ denotes the usual coordinate
        functions on $\R^{2n+1}$.  Let $\rho :
        [0,+\infty)\to [0,+\infty)$ be a smooth non-increasing function supported in
        $[0,1/4]$ such that $\rho(0)=1$. Let $H:\R^{2n+1}\to\R$ be the
        Hamiltonian map defined by $H(p):=\rho(|p|^2)$ for all $p\in\R^{2n+1}$ and
        where $|\cdot|$ denotes the usual Euclidean norm. We denote by $(\varphi_t)$ the
        contact flow generated by $H$.  Finally for all
        $a\in\mathbb{R}$ we denote by $\Phi_a$ the non compactly supported
        contactomorphism of $\mathbb{R}^{2n+1}$ defined as 
        $(x,y,z)\mapsto (e^ax,e^ay,e^{2a}z)$.

        Let $a<0$ be sufficiently close to $0$ so that
        $\Phi_a^{-1}(B_0(1/2))$ is strictly included in $B_0(1)$, where
        $B_0(r)$ denotes the open ball centered at $0$ of radius $r>0$. Then the
        support of $\Phi_a^{-1}\circ \varphi_t\circ\Phi_a$, that is 
        $\Phi_a^{-1}(\supp \varphi_t)$,
        is contained in $B_0(1)$ and contains strictly
        $\supp \varphi_t$ for all $t\in[0,1]$. Moreover since the compactly supported
        Hamiltonian function
        \begin{equation*}
            h_a :\R^{2n+1}\to[0,+\infty),\quad (x,y,z)\mapsto e^{-2a}h(e^ax,e^ay,e^{2a}z),
        \end{equation*}
        generates the contact flow $(\Phi_a^{-1}\varphi_t\Phi_a)$ 
        and trivially satisfies $h_a\geq h$ we deduce that 
        $(\Phi_a^{-1}\varphi_t\Phi_a)$ is a non-negative compactly
        supported contact isotopy, so
        \begin{equation*}
            \varphi_1 \cleq \Phi_a^{-1}\varphi_1\Phi_a\text{ and } \varphi_1\ne
            \Phi_a^{-1}\varphi_1\Phi_a.
        \end{equation*}
        In order to conclude, one needs to replace $\Phi_a$ with a compactly supported
        contactomorphism.
        Since $\Phi_a^{-1}\varphi_t\Phi_a = g^{-1}\varphi_tg$ for any $t\in\R$
        and any diffeomorphism $g$ agreeing with $\Phi_a$ on
        $\Phi_a^{-1}B_0(1/2)$, one can take any such
        $g$ in $\Gcont$.
        Such a $g$ can be induced by a compactly supported Hamiltonian
        map obtained by cutting off the Hamiltonian map generating
        $(\Phi_{ta})_{t\in[0,1]}$.
        Finally,
        the elements $g$ and $\varphi_1$ satisfy the relations \eqref{pas d'invariance},
        as needed.
    \end{proof}

\bibliographystyle{amsplain}
\bibliography{biblio} 

\end{document}